\documentclass[12pt,a4paper]{article}
\pdfoutput=1
\usepackage{amsthm,upref}
\usepackage{mathtools}
\usepackage[theoremfont]{newtxtext}
\usepackage{newtxmath}
\usepackage{microtype}
\usepackage{enumerate}
\usepackage{authblk}
\usepackage[margin=3cm]{geometry}
\usepackage{hyperref,bookmark}
\usepackage{tikz}

\title{Semi-linear evolution equations via positive semigroups}

\author[1]{Wolfgang Arendt}
\author[2]{Daniel Daners}%
\affil[1]{Institut f\"ur Angewandte Analysis, Universit\"at Ulm, Germany\authorcr%
  \nolinkurl{wolfgang.arendt@uni-ulm.de}}%

\affil[2]{School of Mathematics and Statistics, University of Sydney, NSW 2006, Australia\authorcr%
  \nolinkurl{daniel.daners@sydney.edu.au}}%

\date{\today}

\makeatletter
\hypersetup{pdfauthor={Wolfgang Arendt, Daniel Daners}}
\hypersetup{pdftitle={\@title}}
\makeatother

\numberwithin{equation}{section}
\numberwithin{figure}{section}

\DeclareMathOperator{\loc}{loc}

\theoremstyle{plain}
\newtheorem{theorem}{Theorem}[section]
\newtheorem{lemma}[theorem]{Lemma}
\newtheorem{proposition}[theorem]{Proposition}
\newtheorem{corollary}[theorem]{Corollary}

\theoremstyle{definition}
\newtheorem{definition}[theorem]{Definition}

\newtheorem{example}[theorem]{Example}
\newtheorem{examples}[theorem]{Examples}

\theoremstyle{remark}
\newtheorem{remark}[theorem]{Remark}

\DeclareMathOperator{\sign}{sign}
\DeclareMathOperator{\trace}{tr}

\DeclarePairedDelimiter{\tnorm}{|\!|\!|}{|\!|\!|}
\ifdefined\dlb
\DeclarePairedDelimiter{\orderint}{\dlb}{\drb}
\else
\DeclarePairedDelimiter{\orderint}{[\![}{]\!]}
\fi

\let\oldthebibliography\thebibliography
\renewcommand\thebibliography[1]{%
  \oldthebibliography{#1}
  \setlength{\parskip}{.2ex}
  \setlength{\itemsep}{0pt plus 0.3ex}
  \small
}

\begin{document}
\maketitle

\renewcommand{\thefootnote}{}%
\footnotetext{\textbf{Mathematics Subject Classification (2020): 34G20, 47N20, 47H07, 35K20} }%
\footnotetext{\footnotesize\textbf{Keywords:} Semi-linear evolution equation, initial boundary value problems, differential equations in ordered Banach spaces, Lotka-Volterra competition system, logistic equation}%
\begin{center}
  \slshape Dedicated to Professor Herbert Amann on the occasion of his 85th birthday
\end{center}
\begin{abstract}
  We study semi-linear evolutionary problems where the linear part is the generator of a positive $C_0$-semigroup. The non-linear part is assumed to be quasi-increasing. Given an initial value in between a sub- and a super-solution of the stationary problem we find a solution of the semi-linear evolutionary problem. Convergence as $t\to\infty$ is also studied for the solutions. Our results are applied to the logistic equation with diffusion, to a Lotka-Volterra competition model and the Fisher equation from population genetics.
\end{abstract}
\section{Introduction}
The method of sub- and super-solutions is one of the main methods to prove the existence of equilibria for semi-linear elliptic boundary value problems. The use of positive operators on ordered Banach spaces for that purpose was made popular in Amann's seminal paper \cite{amann:76:fpe}. The method can also be used to show the existence of periodic or of quasi-periodic solutions, see for instance \cite{dancer:91:sfp,hess:91:ppb} or the survey \cite{hirsch:05:mmr} including many references and a historical account. The method of sub- and super-solutions not only provides a tool to prove the existence of equilibria, but can also be used to prove the existence of solutions to the corresponding initial value problem. It furthermore allows to establish some convergence and stability results. One particularly fruitful approach in that direction was developed in \cite{sattinger:71:mmn} for parabolic boundary value problems. Our aim is to establish and extend such results to evolutionary problems defined by the non-linear perturbation of the generator of a positive $C_0$-semigroup on an ordered Banach space. This allows us to establish results of existence and uniqueness as well as the asymptotic behaviour of solutions as $t\to\infty$ under minimal regularity assumptions.

Throughout we assume that $E$ is an ordered Banach space with a normal cone $E_+$, see Section~\ref{sec:fixed-point-map} for a definition. A partial order on $E$ is given by $u\leq v$ if and only if $v-u\in E_+$.  Examples of such spaces are $L^p(\Omega)$ for $1\leq p\leq\infty$ or $C(\bar\Omega)$ with $\Omega\subseteq\mathbb R^N$ open and bounded with the order being defined pointwise. Given $\underline u$, $\overline u\in E$ with $\underline u\leq\overline u$ we call
\begin{equation*}%
  [\underline u,\overline u]:=\{u\in E\colon \underline u\leq u\leq\overline u\}
\end{equation*}
an \emph{order interval}. We furthermore assume that $(S(t))_{t\geq 0}$ is a positive $C_0$-semigroup on $E$ with generator $-A$. To say that $(S(t))_{t\geq 0}$ is \emph{positive} means that $S(t)E_+\subseteq E_+$ for all $t\geq 0$. Let $F\in C([\underline u,\overline u],E)$. We study the existence and properties of mild solutions to the semi-linear Cauchy problem
\begin{equation}
  \label{eq:sle-abstract}
  \begin{aligned}
    \dot u(t)+Au(t) & =F(u(t))\qquad \text{for $t>0$,} \\
    u(0)            & =u_0
  \end{aligned}
\end{equation}
with $u_0\in[\underline u,\overline u]$. A mild solution of \eqref{eq:sle-abstract} is a function $u\in C([0,\infty),E)$ such that
\begin{equation}
  \label{eq:mild-solution}
  u(t)=S(t)u_0+\int_0^t S(t-s)F(u(s))\,ds
\end{equation}
for all $t\geq 0$. We will assume that $F\colon [\underline u,\overline u]\to E$ is \emph{quasi-increasing}, which means that for some $\mu\in\mathbb R$, the shifted function $F_\mu$ defined by
\begin{equation}
  \label{eq:quasi-increasing}
  F_\mu(v):=F(v)+\mu v
\end{equation}
is increasing. A function $F\colon [\underline u,\overline u]\to E$ is called \emph{increasing} if $F(v_1)\leq F(v_2)$ for all $v_1,v_2\in[\underline u,\overline u]$ with $v_1\leq v_2$. We furthermore  assume that $\underline u$ and $\overline u$ are weak sub- and super-solutions of the stationary problem
\begin{equation}
  \label{eq:stationary-equation}
  Av=F(v)
\end{equation}
associated with \eqref{eq:sle-abstract}. For a definition we need the \emph{dual cone}
\begin{equation}
  \label{eq:dual-cone}
  E_+':=\{v'\in E'\colon\langle v',v\rangle\geq 0\text{ for all }v\in E_+\}.
\end{equation}
Given the dual operator $A'$ of $A$ we set $D(A')_+:=D(A')\cap E_+'$.
\begin{definition}[sub/super-solution]
  \label{def:sub-equilibria}
  We call $\underline u\in E$ a \emph{weak sub-solution} if $A\underline u\leq F(\underline u)$ weakly and $\overline u\in E$ a \emph{weak super-solution} of \eqref{eq:stationary-equation} if $A\overline u\geq F(\overline u)$ weakly, that is,
  \begin{equation}
    \label{eq:weak-s-equilibria}
    \langle\underline u, A' v'\rangle
    \leq\langle F(\underline u), v'\rangle
    \qquad\text{and}\qquad
    \langle\overline u, A' v'\rangle
    \geq\langle F(\overline u), v'\rangle
  \end{equation}
  for all $v'\in D(A')_+$. If $\underline u\leq\overline u$ we call $\underline u,\overline u$ an \emph{ordered pair of weak sub- and super-solutions} of \eqref{eq:stationary-equation}. We call $u$ a \emph{solution} of \eqref{eq:stationary-equation} or an \emph{equilibrium} of \eqref{eq:sle-abstract} if $v\in D(A)$ and $Av=F(v)$.
\end{definition}
We note that $u\in E$ is an equilibrium if and only if $u$ is a weak sub- and a weak super-solution. In fact, then $\langle u,A'v'\rangle=\langle F(u),v'\rangle$ for all $v'\in D(A')$. Since $A$ is closed and $D(A)$ is dense, this implies that $u\in D(A)$ and $Au=F(u)$. If $v\in E$, then the constant function $u(t):=v$ for all $t\geq 0$ is a solution of~\eqref{eq:stationary-equation} if and only if $v$ is an equilibrium. A function $u\colon [0,\infty)\to E$ is called \emph{increasing} if $u(t)\leq u(s)$ for all $0\leq t\leq s$ and \emph{decreasing} if $u(t)\geq u(s)$ for all $0\leq t\leq s$.

We also sometimes assume that $E$ has order continuous norm. An \emph{order continuous norm} in $E$ means that any increasing (or decreasing) sequence in an order interval converges with respect to the norm in $E$, where a sequence $(u_n)_{n\in\mathbb N}$ is called \emph{increasing} in $E$ if $u_n\leq u_{n+1}$ for all $n\in\mathbb N$. It is called \emph{decreasing} if the inequality is reversed. Examples of ordered Banach spaces with order continuous norm are the $L^p$-spaces with $1\leq p<\infty$. The order continuity comes from the monotone convergence and dominated convergence theorems. If $E$ has order continuous norm, then the cone is normal. The ordered Banach spaces $C(\bar\Omega)$ and $L^\infty(\Omega)$ do not have order continuous norm. The main result of this paper is the following theorem.
\begin{theorem}
  \label{thm:existence}
  Suppose that $E$ is an ordered Banach space with normal cone and let $-A$ be the generator of a positive $C_0$-semigroup $S((t))_{t\geq 0}$ on $E$. Further assume that $\underline u,\overline u\in E$ is an ordered pair of weak sub- and super-solutions of \eqref{eq:stationary-equation} and let $F\in C([\underline u,\overline u] ,E)$ be quasi-increasing. If either $E$ has order continuous norm or $S(t)$ is compact for all $t>0$, then the following assertions hold.
  \begin{enumerate}[\normalfont (i)]
  \item For each initial value $u_0\in[\underline u,\overline u]$ there exists a minimal mild solution $u_{\min}$ and a maximal mild solution $u_{\max}$ of~\eqref{eq:sle-abstract}, that is, any mild solution $u\colon[0,\infty)\to [\underline u,\overline u]$ of \eqref{eq:sle-abstract} with $u(0)=u_0$ satisfies $u_{\min}(t)\leq u(t)\leq u_{\max}(t)$ for all $t\geq 0$.
  \item Let $u_{\min}$ and $u_{\max}$ and $\tilde u_{\min}$ and $\tilde u_{\max}$ be the minimal and maximal mild solutions of \eqref{eq:sle-abstract} with initial values $u_0$ and $\tilde u_0$ in $[\underline u,\overline u]$, respectively. If $u_0\leq\tilde u_0$, then $u_{\min}\leq \tilde u_{\min}$ and $u_{\max}\leq \tilde u_{\max}$.
  \item Denote by $U_{\min}$ the minimal mild solution with $u(0)=\underbar u$ and by $U_{\max}$ the maximal mild solution with $u(0)=\overline u$ of \eqref{eq:sle-abstract}. Then for every $u_0\in[\underline u,\overline u]$ and every mild solution $u\colon[0,\infty)\to [\underline u,\overline u]$ of \eqref{eq:sle-abstract} with $u(0)=u_0$
    \begin{equation*}
      U_{\min}(t)\leq u(t)\leq U_{\max}(t)
    \end{equation*}
    for all $t\geq 0$.
  \item The function $U_{\min}\in C([0,\infty),E)$ is increasing and $U_{\max}\in C([0,\infty),E)$ is decreasing. Moreover,
    \begin{equation*}
      u_*:=\lim_{t\to\infty}U_{\min}(t)
      \qquad\text{and}\qquad
      u^*:=\lim_{t\to\infty}U_{\max}(t)
    \end{equation*}
    exist and $u_*$ and $u^*$ are the minimal and maximal solutions of \eqref{eq:stationary-equation} in $[\underline u,\overline u]$.
  \end{enumerate}
\end{theorem}
We emphasise that $F\colon[\underline u,\overline u]\to E$ is only assumed to be continuous, so we cannot expect the uniqueness of solutions for \eqref{eq:sle-abstract}. The above theorem shows that for any $u_0\in [\underline u,\overline u]$ there exists a minimal and a maximal solution. If $F$ is locally Lipschitz in $[\underline u,\overline u]$, then the solutions to any given initial value turns out to be unique, that is, $u_{\min}=u_{\max}$, but not necessarily otherwise, see Section~\ref{sec:uniqueness}.

We continue by making some remarks about the assumptions on the non-linearities discussed above.

\begin{remark}
  \label{rem:superposition-operator}
  (a) Let $\Omega\subseteq\mathbb R^N$ be a bounded and open. Let $E=L^p(\Omega)$ with $1\leq p<\infty$. The non-linearity $F$ is typically a substitution operator on $L^p(\Omega)$ associated with a function $f\in C(\overline\Omega\times\mathbb R)$ that is Lipschitz continuous on bounded sets of $\mathbb R$ uniformly with respect to $x\in\bar\Omega$. This means that for every bounded interval $[m_1,m_2]\subseteq\mathbb R$ there exists $L>0$ such that
  \begin{equation}
    \label{eq:f-lipschitz}
    |f(x,\xi_2)-f(x,\xi_1)|\leq L|\xi_2-\xi_1|
  \end{equation}
  for all $\xi_1,\xi_2\in[m_1,m_2]$ and all $x\in\bar\Omega$. We define the corresponding substitution operator by
  \begin{equation}
    \label{eq:substitution-operator}
    [F(u)](x):=f(x,u(x))
  \end{equation}
  for every function $u\colon\Omega\to\mathbb R$ and $x\in\Omega$. Asking that $F\colon L^p(\Omega)\to L^p(\Omega)$ is Lipschitz continuous, or even just a function between those spaces is a very strong condition. It implies that $f$ be of at most linear growth in $\xi\in\mathbb R$, see for instance \cite{appell:90:nso}. Even simple non-linearities such as the logistic growth $au-mu^2$ do not fulfil this condition. The way out is that generally the sub- and super-solutions are in $L^\infty(\Omega)$ and hence the restriction of $F$ to the order interval $[\underline u,\overline u]$ fulfils the Lipschitz condition.

  (b) Generally $f$ is not increasing, but if the order interval is bounded in $L^\infty(\Omega)$, then $f$ is quasi-increasing on that order interval. Indeed, let $m_1:=\inf_{x\in\Omega}\underline u(x)$, $m_2:=\sup_{x\in\Omega}\overline u(x)$ and $\mu:=L$. Then \eqref{eq:f-lipschitz} implies that for $m_1\leq\xi_1\leq\xi_2\leq m_2$
  \begin{equation*}
    f(x,\xi_2)-f(x,\xi_2)\geq -L(\xi_2-\xi_1)
  \end{equation*}
  for all $x\in\bar\Omega$. Hence $\xi\mapsto f(x,\xi)+L\xi$ is increasing on $[m_1,m_2]$ for all $x\in\Omega$ and thus the corresponding substitution operator is quasi-increasing on $[\underline u,\overline u]$. This is a condition that first seems to appear in \cite{amann:72:ems}.
\end{remark}
\section{The fixed point map and mild solutions}
\label{sec:fixed-point-map}
We saw in the introduction that mild solutions are solutions of the integral equation \eqref{eq:mild-solution}. That integral equation can be seen as a fixed point equation. In this section we study properties of this fixed point map.

We start by introducing some terminology. Let $Z$ be a vector space. A subset $Z_+$ of $Z$ is called a \emph{cone} if $Z_++Z_+\subseteq Z_+$ and $[0,\infty)Z_+\subseteq Z_+$. The cone is called \emph{proper} if $Z_+\cap (-Z_+)=\{0\}$. An \emph{ordered vector space} is a vector space $Z$ with proper cone $Z_+$. Then $u\leq v$ if and only if $v-u\in Z_+$ defines a partial order on $Z$. An \emph{ordered Banach space} is an ordered vector space $Z$ with a complete norm such that the positive cone $Z_+$ is closed. Note that order intervals in $Z$ are convex and closed. The cone $E_+$ is called \emph{normal} if all order intervals are norm bounded. The spaces $L^p(\Omega)$, $1\leq p\leq \infty, $with $\Omega\subseteq\mathbb R^N$ open, and $E=C(\bar\Omega)$ for $\Omega\subseteq\mathbb R^N$ bounded, are ordered Banach spaces with normal cone.

Recall that $E$ is an ordered Banach space with normal cone and $S((t))_{t\geq 0}$is a positive $C_0$-semigroup on $E$ with generator $-A$. We start by characterising and justifying the term ``mild solution'' for \eqref{eq:mild-solution} used in the introduction. We note that the order structure is irrelevant for that, it holds for an arbitrary $C_0$-semigroup on a Banach space.
\begin{proposition}
  \label{prop:weak-solution-characterisation}
  Suppose that $u_0\in E$ and that $u,f\in C([0,\infty),E)$. Then the following statements are equivalent.
  \begin{enumerate}[\normalfont (i)]
  \item For all $t\geq 0$
    \begin{equation}
      \label{eq:w-fixed-point}
      u(t)=S(t)u_0+\int_0^tS(t-s)f(s)\,ds;%
    \end{equation}
  \item $u(0)=u_0$, $[t\mapsto \langle v',u(t)\rangle]\in C^1([0,\infty))$ and
    \begin{equation*}
      \frac{d}{dt}\langle v',u(t)\rangle+\langle A'v',u(t)\rangle
      =\langle v',f(t)\rangle
    \end{equation*}
    for all $t\geq 0$ and all $v'\in D(A')$.
  \end{enumerate}
\end{proposition}
\begin{proof}
  As $-A'$ is the weak$^*$ generator of $S(t)'$, for $v'\in D(A')$ and $v\in E$ we have $\langle v',S(\cdot)v\rangle\in C^1([0,\infty))$. Moreover,
  \begin{equation}
    \label{eq:Sdual-weak-derivative}
    \frac{d}{dt}\langle v',S(t)v\rangle+\langle A'v',S(t)v\rangle=0
  \end{equation}
  for all $t\geq 0$, see \cite[Example~II.2.5]{engel:00:ops}

  (i) $\implies$ (ii): Let $v'\in D(A')$. Then by (i) we have
  \begin{equation*}
    \langle v',u(t)\rangle
    =\langle v',S(t)u_0\rangle+\int_0^t\langle v',S(t-s)f(s)\rangle\,ds.
  \end{equation*}
  Using \eqref{eq:Sdual-weak-derivative} we obtain
  \begin{align*}
    \frac{d}{dt}\langle v',u(t)\rangle
     & =-\langle A'v',S(t)u_0\rangle
    +\langle v',f(t)\rangle -\int_0^t\langle A'v',S(t-s)f(s)\rangle\,ds \\
     & =-\langle A'v',u(t)\rangle+\langle v',f(t)\rangle
  \end{align*}
  for all $t\geq 0$, proving (ii).

  (ii) $\implies$ (i): Assume that $u\in C([0,\infty),E)$ satisfies (ii) and set
  \begin{equation*}
    v(t):=S(t)u_0+\int_0^tS(t-s)f(s)\,ds.
  \end{equation*}
  Then $w:=u-v\in C([0,\infty),E)$, $w(0)=0$ and
  \begin{equation}
    \label{eq:w-derivative}
    \frac{d}{dt}\langle v',w(t)\rangle+\langle A'v',w(t)\rangle=0
  \end{equation}
  for all $v'\in D(A')$. Let
  \begin{equation*}
    W(t):=\int_0^tw(s)\,ds
  \end{equation*}
  for all $t\geq 0$. Then, by using \eqref{eq:w-derivative} we have
  \begin{equation*}
    \langle A'v',W(t)\rangle
    =\int_0^t\langle A'v',w(s)\rangle\,ds
    =-\int_0^t\frac{d}{ds}\langle v',w(s)\rangle\,ds
    =-\langle v',w(t)\rangle.
  \end{equation*}
  This implies that $W(t)\in D(A)$ and $AW(t)=-w(t)=\dot W(t)$ for all $t\geq 0$. Since $W(0)=0$, it follows that $W=0$ and hence $w=u-v=0$ as well. In particular $u=v$, proving (i).
\end{proof}
In what follows we consider the Fréchet space
\begin{equation*}
  Z:=L_{\loc}^1([0,\infty),E).
\end{equation*}
Convergence in $Z$ is defined by $u_n\to u$ in $Z$ if and only if
\begin{equation*}
  \lim_{n\to\infty}\int_0^T\|u_n(t)-u(t)\|_E\,dt= 0
\end{equation*}
for all $T>0$. We note that $Z$ is an ordered vector space with the closed cone
\begin{equation*}
  Z_+:=\{u\in Z\colon u(t)\geq 0\text{ a.~e.}\}.
\end{equation*}
Let $[\underline u,\overline u]$ be an order interval in $E$ and assume that $F\in C([\underline u,\overline u] ,E)$ is increasing, that is, $\underline u\leq u_1\leq u_2\leq\overline u$ implies $F(u_1)\leq F(u_2)$. We set
\begin{equation*}
  L_{\loc}^1([0,\infty),[\underline u,\overline u])
  :=\{u\in L_{\loc}^1([0,\infty),[\underline u,\overline u])\colon \underline u\leq u(t)\leq\overline u\text{ a.~e.}\}.
\end{equation*}
As $F$ is increasing we have that $F(v)\in[F(\underline u),F(\overline u)]$ for all $v\in[\underline u,\overline u]$ and since order intervals are bounded in $E$ it follows that
\begin{equation}
  \label{eq:image-bounded-C}
  C:=\sup\left\{\|F(v)\|\colon \underline u\leq v\leq\overline u\right\}<\infty.
\end{equation}
Thus, given $u\in L_{\loc}^1([0,\infty),E)$, the function $s\mapsto S(t-s)F(u(s))$ is bounded and measurable and hence Bochner integrable on $(0,t)$. Given any initial value $u_0\in[\underline u,\overline u]$ and $u\in L_{\loc}^1([0,\infty),[\underline u,\overline u])$ we can therefore define the \emph{fixed point map associated with $u_0$} by
\begin{equation}
  \label{eq:fixed-point-map-u0}
  G(u)(t):=S(t)u_0+\int_0^tS(t-s)F(u(s))\,ds
\end{equation}
for all $t\geq 0$. Looking at \eqref{eq:mild-solution} we see that $u\in C([0,\infty),E)$ is a mild solution of \eqref{eq:sle-abstract} if and only if it is a fixed point of $G$.  We need some properties of the map $G$.
\begin{lemma}
  \label{lem:Gu-continuous}
  Let $G$ be defined by \eqref{eq:fixed-point-map-u0}. Then
  \begin{equation*}
    G\colon L_{\loc}^1\left([0,\infty),[\underline u,\overline u]\right)\to C([0,\infty),E).
  \end{equation*}
  Moreover, if $F$ is increasing, then $G$ is increasing.
\end{lemma}
\begin{proof}
  As $(S(t))_{t\geq 0}$ is a positive semigroup and $F$ is increasing, it follows that $G$ is increasing. It remains to show the continuity of $G(u)$ for $u\in L^1([0,\infty),[\underline u,\overline u])$. Let $u\in L^1([0,\infty),[\underline u,\overline u])$. Let $t_n\downarrow t$ in $[0,\infty)$. Then
  \begin{align*}
    G(u)(t_n) & -G(u)(t)                              \\
              & =\int_0^t(S(t_n-s)-S(t-s))F(u(s))\,ds
    +\int_t^{t_n}S(t_n-s)F(u(s))\,ds
  \end{align*}
  If $t_n\uparrow t$, then
  \begin{align*}
    G(u)(t_n) & -G(u)(t)                                                         \\
              & =\int_0^{t}1_{[0,t_n]}(s)\left(S(t_n-s)-S(t-s)\right)F(u(s))\,ds
    +\int_{t_n}^tS(t-s)F(u(s))\,ds
  \end{align*}
  In either case, by \eqref{eq:image-bounded-C} the integrands are uniformly bounded with respect to $n\in\mathbb N$. Thus, the first integral converges to zero as $n\to\infty$ by the dominated convergence theorem; see \cite[Corollary~1.1.8]{arendt:11:vlt}. The second integral converges to zero as $n\to\infty$ by an obvious estimate.
\end{proof}

We now assume that $\underline u,\overline u\in E$ is a pair of weak sub- and super-solutions of \eqref{eq:stationary-equation} as given in Definition~\ref{def:sub-equilibria}. We now consider the order interval
\begin{equation}
  \label{eq:interval-w}
  \orderint{\underline u,\overline u}
  :=\left\{u\in L_{\loc}^1([0,\infty),E)\colon \underline u\leq u(t)\leq\overline u\text{ a.e.}\right\},
\end{equation}
in $L_{\loc}^1([0,\infty),E)$. In particular this means that we use the topology on $L_{\loc}^1([0,\infty),E)$ when looking at convergence in that order interval. This order interval turns out to be invariant under $G$. We will also need properties of the functions
\begin{equation}
  \label{eq:sub-solution-uw}
  \underline w(t)
  =S(t)\underline u+\int_0^tS(t-s)F(\underline u)\,ds
  =S(t)\underline u+\int_0^tS(s)F(\underline u)\,ds
\end{equation}
and
\begin{equation}
  \label{eq:sub-solution-ow}
  \overline w(t)=S(t)\overline u+\int_0^tS(s)F(\overline u)\,ds.
\end{equation}
The following lemma holds.
\begin{lemma}
  \label{lem:G-invariant}
  Let $\underline u,\overline u$ be a pair of weak sub- and super-solutions for \eqref{eq:stationary-equation} and let $F\in C([\underline u,\overline u],E)$ be increasing. Define $\underline w$ and $\overline w$ by \eqref{eq:sub-solution-uw} and \eqref{eq:sub-solution-ow}. Then $\underline w(0)=\underline u$ and $\underline w$ is increasing on $[0,\infty)$. Similarly,  $\overline w(0)=\overline u$ and $\overline w$ is decreasing on $[0,\infty)$. Furthermore, $G(\underline u)\geq\underline u$ and $G(\overline u)\leq\overline u$. Finally,
  \begin{equation*}
    G\colon \orderint{\underline u,\overline u}\to \orderint{\underline u,\overline u}
  \end{equation*}
  is continuous, where $\orderint{\underline u,\overline u}$ carries the topology of $L_{\loc}^1([0,\infty),E)$.
\end{lemma}
\begin{proof}
  By definition $\underline w(0)=\underline u$. Let $v'\in D(A')_+$. By the second part of \eqref{eq:sub-solution-uw} and \eqref{eq:weak-s-equilibria}
  \begin{align*}
    \frac{d}{dt}\langle v',\underline w(t)\rangle
    &=\frac{d}{dt}\langle v',S(t)\underline u\rangle+\frac{d}{dt}\left\langle v',\int_0^tS(s)F(\underline u)\right\rangle\\
    &=\langle -A'v',S(t)\underline u\rangle+\langle v',S(t)F(\underline u)\rangle\\
    &=\langle -A'S(t)'v',\underline u\rangle+\langle S(t)'v',F(\underline u)\rangle
    \geq 0
  \end{align*}
  since $0\leq S(t)'v'\in D(A')$. As $D(A')_+$ determines positivity by Corollary~\ref{cor:weak-positivity} in the appendix it follows that $\underline w$ is increasing. A similar argument shows that $\overline w$ is decreasing with initial value $\overline u$. Using that $\underline w(t)\geq\underline u$ for all $t\geq 0$ by what we just proved, we see that
  \begin{align*}
    G(\underline u)(t) & =S(t)u_0+\int_0^tS(t-s)F(\underline u)\,ds              \\
                       & \geq S(t)\underline u+\int_0^tS(t-s)F(\underline u)\,ds
    =\underline w(t)\geq\underline u
  \end{align*}
  for all $t\geq 0$. A similar argument shows that $G(\overline u)\leq\overline u$. As $G$ is increasing $\orderint{\underline u,\overline u}$ is invariant under $G$.

  To prove the continuity let $u_n\in\orderint{\underline u,\overline u}$ with $u_n\to u$ in $L_{\loc}^1([0,\infty),E)$ as $n\to\infty$. Let $T>0$. We have to show that $G(u_n)\to G(u)$ in $L^1([0,T],E)$. There exists a sub-sequence that $u_{n_k}(t)\to u(t)$ in $E$ for almost every $t\in[0,T]$. As $F$ is continuous, $F(u_{n_k}(s))\to F(u(s))$ almost everywhere on $(0,T)$ as $k\to\infty$. There exists $M\geq 1$ such that $\|S(t)\|\leq M$ for all $t\in [0,T]$. As $G(u_n)\in\orderint{\underline u,\overline u}$ we also know that $F(\underline u)\leq F(u_{n_k}(s))\leq F(\overline u)$ for all $s\in[0,T]$. Thus, by the normality of the cone there exists $c\geq 0$ with $\|F(u_{n_k}(s)\|\leq c$ for all $s\in[0,T]$ and all $k\in\mathbb N$. The dominated convergence theorem implies that $G(u_{n_k})(t)\to G(u)(t)$ in $E$ for all $t\geq 0$. Applying the dominated convergence theorem again we see that $G(u_{n_k})\to G(u)$ in $L^1([0,T],E)$. Since each sub-sequence has a sub-sequence that converges to $G(u)$ we deduce that $G(u_n)\to G(u)$ in $L^1([0,T],E)$ as $n\to\infty$.
\end{proof}
The next lemma reflects the autonomous nature of the problem.
\begin{lemma}[Translation of mild solution]
  \label{lem:shift}
  Suppose that $u\in C([0,\infty),E)$ is a mild solution of \eqref{eq:sle-abstract}. Fix $t_0>0$ and define $v(t):=u(t_0+t)$ for all $t\geq 0$. Then $v\in C([0,\infty),E)$ is a mild solution of \eqref{eq:sle-abstract} with initial value $v(0)=u(t_0)$.
\end{lemma}
\begin{proof}
  As $u$ is a mild solution of \eqref{eq:sle-abstract} we have that
  \begin{align*}
    v(t)&=u(t+t_0)=S(t_0+t)u_0+\int_0^{t+t_0}S(t_0+t-s)F(u(s))\,ds\\
        &=S(t)\left[S(t_0)u_0+\int_0^{t_0}S(t_0-s)F(u(s))\,ds\right]+\int_{t_0}^{t_0+t}S(t_0+t-s)F(u(s))\,ds\\
        &=S(t)v(0)+\int_0^tS(t-s)F(u(t_0+s))\,ds\\
        &=S(t)v(0)+\int_0^tS(t-s)F(v(s))\,ds
  \end{align*}
  for all $t>0$ and thus $v$ is a mild solution of \eqref{eq:sle-abstract} with initial condition $v(0)=u(t_0)$.
\end{proof}
We call the semigroup $(S(t))_{t\geq 0}$ \emph{compact} if $S(t)$ is a compact operator on $E$ for all $t>0$. This implies that $t\mapsto S(t)$, $(0,\infty)\to \mathcal L(E)$ is continuous with respect to the operator norm, see for instance \cite[Lemma~II.4.22]{engel:00:ops}. We will show that $G$ inherits that compactness.

A mapping $B\colon \orderint{\underline u,\overline u}\to Z:=L_{\loc}^1([0,\infty),E)$ is called \emph{compact} if for every sequence $(u_n)_{n\in\mathbb N}$ in $\orderint{\underline u,\overline u}$ there exists a sub-sequence $(u_{n_k})_{k\in\mathbb N}$ such that $\left(B(u_{n_k})\right)_{k\in\mathbb N}$ converges in $Z$ as $k\to\infty$.
\begin{lemma}[Compactness of the fixed point map]
  \label{lem:G-compact}
  Assume that $(S(t))_{t\geq 0}$ is compact. Let $u_0\in[\underline u,\overline u]$ and let $G\colon\orderint{\underline u,\overline u}\to\orderint{\underline u,\overline u}$ be the associated fixed point map given by \eqref{eq:fixed-point-map-u0}. Then $G$ is compact.
\end{lemma}
\begin{proof}
  (I) We fix $T>0$ and define
  \begin{equation*}
    K:=\left\{G(u)|_{[0,T]}\colon u\in\orderint{\underline u,\overline u}\right\}.
  \end{equation*}
  we show that $K$ is relatively compact in $L^1([0,T],E)$. We do that in two parts. For fixed $\delta\in(0,T)$ we first consider
  \begin{equation*}
    G_\delta\colon\orderint{\underline u,\overline u}\to C([0,T],E)
  \end{equation*}
  given by
  \begin{equation*}
    G_\delta(u)(t):=
    \begin{cases}
      S(t)u_0+\int_\delta^tS(t-s)F(u(s))\,ds & \text{if }t\in(\delta,T] \\
      S(t)u_0                                & \text{if }t\in[0,\delta]
    \end{cases}
  \end{equation*}
  for all $u\in\orderint{\underline u,\overline u}$. We claim that
  \begin{equation*}
    K_\delta:=\left\{G_\delta(u)\colon u\in\orderint{\underline u,\overline u}\right\}
  \end{equation*}
  is compact in $C([0,T],E)$. For that we use the Arzelà-Ascoli theorem for vector valued functions and show that $G_\delta$ is pointwise relatively compact and equi-continuous; see \cite[Theorem~XII.6.4]{dugundji:66:top}. For the pointwise compactness we need to show that for any given $t\in[0,T]$ the set
  \begin{equation*}
    B_t:=\left\{G_\delta(t)\colon u\in\orderint{\underline u,\overline u}\right\}
  \end{equation*}
  is relatively compact in $E$. If $t\in[0,\delta]$, then $G_\delta(u)(t)=S(t)u_0$, so $B_t=\{S(t)u_0\}$ is compact. If $t\in(\delta,T]$, then
  \begin{equation*}
    G_\delta(u)(t)=S(\delta)\left(S(t-\delta)u_0+\int_\delta^tS(t-s-\delta)F(u(s))\,ds\right).
  \end{equation*}
  Since the set
  \begin{equation*}
    \left\{S(t-\delta)u_0+\int_\delta^tS(t-s-\delta)F(u(s))\,ds
    \colon u\in\orderint{\underline u,\overline u}\right\}\subseteq E
  \end{equation*}
  is bounded and $S(\delta)$ is a compact operator, it follows that $B_t$ is relatively compact as well. We now show that the set $K_\delta$ is equi-continuous at each $t_0\in[0,T]$. Let $t_0\in[0,\delta]$. As $G_\delta(u)=S(\cdot)u_0$ is continuous on the compact interval $[0,\delta]$ it follows that $K_\delta$ is equi-continuous at each $t_0\in[0,\delta)$ and equi-continuous from the left at $t_0=\delta$. To deal with $t_0\in[\delta,T]$ note that there exists $c>0$ such that $\|S(r)F(v)\|\leq c$ for all $v\in[\underline u,\overline u]$ and $r\in[0,T]$. Hence, if $t_0\in[\delta,T)$, $t\in[t_0,T]$ and $u\in\orderint{\underline u,\overline u}$, then
  \begin{align*}
    \|G_\delta(u)(t) & -G(u)(t_0)\|                                                    \\
                     & =\left\|\int_\delta^{t_0}\bigl(S(t-s)-S(t_0-s)\bigr)F(u(s))\,ds
    +\int_{t_0}^tS(t-s)F(u(s))\,ds\right\|                                             \\
                     & \leq Tc\sup_{s\in[\delta,T]}\|S(t-s)-S(t_0-s)\|+c|t-t_0|
  \end{align*}
  If $t_0\in(\delta,T]$, $t\in[\delta,t_0)$ and $u\in\orderint{\underline u,\overline u}$, then similarly
  \begin{align*}
    \|G_\delta(u)(t) & -G(u)(t_0)\|                                                  \\
                     & =\left\|\int_\delta^{t}\bigl(S(t-s)-S(t_0-s)\bigr)F(u(s))\,ds
    +\int_t^{t_0}S(t-s)F(u(s))\,ds\right\|                                           \\
                     & \leq Tc\sup_{s\in[\delta,T]}\|S(t-s)-S(t_0-s)\|+c|t-t_0|
  \end{align*}
  Since $S\colon[\delta,T]\to\mathcal L(E)$ is uniformly continuous, $K_\delta$ is equi-continuous from the right and from the left for every $t_0\in[\delta,T]$. Hence $K_\delta$ is equi-continuous for every $t_0\in[0,T]$, proving that $K_\delta$ is relatively compact in $C([0,T],E)$.

  We next show that $K$ is relatively compact in $L^1([0,T],E)$. For that it is sufficient to show that $K$ is totally bounded, that is, for each $\varepsilon>0$, the set $K$ can be covered by finitely many balls of radius $\varepsilon$. Fix $0<\varepsilon<T$. By the previous part of the proof there exists a family $(v_j){j=1,\dots,n}$ in $C([0,T],E)$ such that for every $u\in\orderint{\underline u,\overline u}$ there exists $j\in\{1,\dots,n\}$ such that
  \begin{equation*}
    \left\|G_\varepsilon(u)(t)-v_j(t)\right\|<\varepsilon
  \end{equation*}
  for all $t\in[0,T]$. Then
  \begin{align*}
    \int_0^T & \left\|G(u)(t)-v_j(t)\right\|\,dt                                    \\
             & \leq\int_0^T\left\|G(u)(t)-G_\varepsilon(u)(t)\right\|\,dt
    +\int_0^T\left\|G_\varepsilon(u)(t)-v_j(t)\right\|\,dt                          \\
             & \leq \int_0^T\left\|\int_0^\varepsilon S(t-s)F(u(s))\,ds\right\|\,dt
    +\int_0^T\left\|G_\varepsilon(u)(t)-v_j(t)\right\|\,dt                          \\
             & < cT\varepsilon+T\varepsilon.
  \end{align*}
  This shows that $K$ can be covered by finitely many balls of radius $T(c+1)\varepsilon$ and thus $K$ is relatively compact in $L^1([0,T],E)$ for every $T>0$.

  (II) Let now $(u_n)_{n\in\mathbb N}$ be a sequence in $\orderint{\underline u,\overline u}$. According to (I), for each $m\in\mathbb N$ the sequence $\left(G(u_n)\right)_{n\in\mathbb N}$ has a convergent sub-sequence in $L^1([0,m],E)$. By Cantor's diagonal argument we find a sub-sequence which converges in $L^1([0,m],E)$ for each $m\in\mathbb N$. This sub-sequence converges in $Z$ and thus $G$ is compact.
\end{proof}

\section{Monotone iterations and convergence}
\label{sec:monotone-iterations}
As before, let $E$ be an ordered Banach space with normal cone and let $-A$ be the generator of a positive $C_0$-semigroup $(S(t))_{t\geq 0}$ on $E$. Let $\underline u,\overline u$ be an ordered pair of weak sub- and super-solutions of \eqref{eq:stationary-equation} and let $F\in C([\underline u,\overline u],E)$ be increasing. Define $\underline w$ and $\overline w$ by \eqref{eq:sub-solution-uw} and \eqref{eq:sub-solution-ow}, respectively. Let $u_0\in[\underline u,\overline u]$. We have proved in Lemma~\ref{lem:G-invariant} that $\orderint{\underline u,\overline u}\subseteq L_{\loc}^1([0,\infty),E)$ is invariant under the fixed point map $G$ \eqref{eq:fixed-point-map-u0} associated with $u_0$ as given in \eqref{eq:fixed-point-map-u0}. Hence the following definition makes sense.
\begin{definition}[Upper/lower iteration sequences]
  \label{def:interated-sequences}
  Let $u_0\in[\underline u,\overline u]$ and let $G$ be the fixed point map associated with $u_0$ as given in \eqref{eq:fixed-point-map-u0}. Inductively define
  \begin{equation*}
    \underline w_0:=\underline u,\quad \underline w_{n+1}:=G(\underline w_n)
    \qquad\text{and}\qquad
    \overline w_0:=\overline u,\quad \overline w_{n+1}:=G(\overline w_n)
  \end{equation*}
  for all $n\in\mathbb N$. We call $(\underline w_n)_{n\in\mathbb N}$ the \emph{lower iteration sequence} and $(\overline w_n)_{n\in\mathbb N}$ the \emph{upper iteration sequence} associated with $u_0$.
\end{definition}
We next collect some properties of these iteration sequences.
\begin{proposition}[Monotone iterations]
  \label{prop:iterated-sequences}
  Let $\underline u,\overline u$ be an ordered pair of weak sub- and super-solutions of \eqref{eq:stationary-equation}. Fix $u_0\in[\underline u,\overline u]$ and let $G$ be the associated fixed point map given by \eqref{eq:fixed-point-map-u0}. Let $(\underline w_n)_{n\in\mathbb N}$ and $(\overline w_n)_{n\in\mathbb N}$ be the lower and upper iteration sequences associated with $u_0$ as in Definition~\ref{def:interated-sequences}. Then
  \begin{equation}
    \label{eq:iteration-comparison}
    \underline u\leq \underline w_n\leq\underline w_{n+1}\leq\overline w_{m+1}\leq\overline w_{m+1}\leq\overline u
  \end{equation}
  for all $n,m\geq 1$.
\end{proposition}
\begin{proof}
  We first prove by induction that $(\underline w_n)_{n\in\mathbb N}$ is an increasing sequence bounded from above by $\overline w$. For the start of the induction note that by Lemma~\ref{lem:G-invariant} and the monotonicity of $G$ we have that
  \begin{equation*}
    \underline w_0=\underline u\leq G(\underline u)=\underline w_1\leq G(\overline u)\leq \overline u.
  \end{equation*}
  Assuming that $\underline w_n\leq\underline w_{n+1}\leq\overline u$ for some $n\geq 0$ we deduce from the monotonicity of $G$ that
  \begin{equation*}
    \underline w_{n+1}=G(\underline w_n)\leq G(\underline w_{n+1})=\underline w_{n+2}\leq G(\overline u)\leq\overline u
  \end{equation*}
  as claimed. Now fix $n\geq 0$. We prove by induction that $(\overline w_m)_{m\in\mathbb N}$ is a decreasing sequence bounded from below by $\underline w_{n+1}$. The start of the induction follows from Lemma~\ref{lem:G-invariant}, the monotonicity of $G$ and the fact that $\underline w_n\leq\overline w$ which imply that
  \begin{equation*}
    \overline w_0=\overline u\geq G(\overline u)=\overline w_1\geq G(\underline w_n)=\underline w_{n+1}.
  \end{equation*}
  For the induction step assume that $\overline w_m\geq\overline w_{m+1}\geq \underline w_{n+1}$. Then also $\overline w_{m+1}\geq \underline w_n$ and thus
  \begin{equation*}
    \overline w_{m+1}=G(\overline w_m)\geq G(\overline w_{m+1})=\overline w_{m+2}\geq G(\underline w_n)\geq\underline w_{n+1}
  \end{equation*}
  as claimed.
\end{proof}
We will show the convergence of the iterated sequences under two conditions. One is the compactness of the semigroup that implies the compactness of the fixed point map $G$ by Lemma~\ref{lem:G-compact}. The other is a condition on the ordered Banach space.

The key to convergence under the assumption of compactness is the following fact on monotone sequences.
\begin{lemma}[Convergence of montone sequences]
  Let $u_n,u\in L_{loc}^1([0,\infty),E)$ be such that $u_n\leq u_{n+1}$ for all $n\in\mathbb N$. If there exists a sub-sequence $(u_{n_k})_{k\in\mathbb N}$ converging to $u$, then $u_n\to u$ in $L_{\loc}^1([0,\infty),E)$ as $n\to\infty$.
\end{lemma}
\begin{proof}
  Since the positive cone on $E$ is normal there exists an equivalent norm $\tnorm{{\cdot}}$ on $E$ which is monotone, that is, $0\leq v_1\leq v_2$ implies $\tnorm{v_1}\leq\tnorm{v_2}$; see Lemma~\ref{lem:normal-cone} in the Appendix. Let $T>0$ and fix $\varepsilon>0$. There exists $n_0$ such that
  \begin{equation*}
    \int_0^T\tnorm{u(t)-u_{n_0}(t)}\,dt<\varepsilon.
  \end{equation*}
  Then, $0\leq u(t)-u_n(t)\leq u(t)-u_{n_0}(t)$ for all $t\in[0,T]$ and all $n\geq n_0$. By the monotonicity of the norm
  \begin{equation*}
    \int_0^T\tnorm{u(t)-u_{n}(t)}\,dt
    \leq \int_0^T\tnorm{u(t)-u_{n_0}(t)}\,dt<\varepsilon
  \end{equation*}
  for all $n\geq n_0$. Hence $u_n\to u$ in $L^1([0,T],E)$ for all $T>0$ as claimed.
\end{proof}

We next turn to the second condition of interest, which is a condition on the ordered Banach space $E$. We call a sequence $(v_n)_{n\in\mathbb N}$ in $E$ \emph{order bounded from above} if there exists $v\in E$ such that $v_n\leq v$ for all $n\in\mathbb N$. The sequence is called \emph{order bounded from below} if there exists $v\in E$ with $v\leq v_n$ for all $n\in\mathbb N$. We call the sequence \emph{order bounded} if it is order bounded from above and from below. We are interested in the convergence of order bounded monotone sequences.
\begin{definition}[order continuous norm]
  The ordered Banach space $E$ is said to have \emph{order continuous norm} if each order bounded increasing sequence converges in $E$.
\end{definition}
We note that if $E$ has order continuous norm, then also each decreasing sequence converges in $E$ if it is order bounded from below.
\begin{remark}
  (a) It is not difficult to show that the order continuity of the norm in an ordered Banach space implies that the positive cone is normal.

  (b) If $\emptyset\neq\Omega\subseteq\mathbb R^N$ is open, then $E=L^p(\Omega)$ has order continuous norm if $1\leq p<\infty$, but $L^\infty(\Omega)$ does not. Neither does $C(\bar\Omega)$ have order continuous norm if $\Omega$ is bounded.
\end{remark}
We need the following lemma on the convergence of monotone sequences in $\orderint{\underline u,\overline u}$ given that $E$ has order continuous norm.
\begin{lemma}
  \label{lem:monotone-norm-convergence}
  Let $(u_n)_{n\in\mathbb N}$ be a sequence in $\orderint{\underline u,\overline u}\subseteq L^1([0,\infty),E)$ with $u_n\leq u_{n+1}$ for all $n\in\mathbb N$. If $E$ has order continuous norm, then $(u_n)$ converges in $L^1([0,\infty),E)$.
\end{lemma}
\begin{proof}
  By assumption $u(t):=\lim u_n(t)$ exists for every $t\geq 0$. Since the cone is normal there exists $c\geq 0$ with $\|v\|\leq c$ for all $v\in[\underline u,\overline u]$. Thus
  \begin{equation*}
    \|u_n(t)-u(t)\|\leq 2c
  \end{equation*}
  for all $n\in\mathbb N$ and $t\geq 0$. It follows from the dominated convergence theorem that
  \begin{equation*}
    \lim_{n\to\infty}\int_0^T\|u_n(t)-u(t)\|\,dt=0
  \end{equation*}
  for all $T>0$. Hence $u_n\to u$ in $L^1([0,\infty),E)$.
\end{proof}

\section{Existence and comparison of mild solutions}
\label{sec:existence}
In this section we prove the bulk of claims in the main Theorem~\ref{thm:existence}. using the facts established in the previous section. We work under the assumptions of that theorem. We let $E$ be an ordered Banach space with normal cone and let $-A$ be the generator of a positive $C_0$-semigroup $(S(t))_{t\geq 0}$ on $E$. We assume that $\underline u,\overline u\in E$ is a pair of ordered sub- and super-solutions of \eqref{eq:stationary-equation}. Assume that $F\in C([\underline u,\overline u],E)$ is quasi-increasing and that $\underline u,\overline u\in E$ are a pair of ordered sub- and super-solutions of \eqref{eq:stationary-equation}.  By a solution to \eqref{eq:sle-abstract} we always mean a mild solution, that is, $u\in C([0,\infty),E)$ satisfying \eqref{eq:mild-solution}.

Before we start the proof of Theorem~\ref{thm:existence} we give a scaling argument that allows us to assume without loss of generality that $F$ is increasing.
\begin{lemma}[Scaling]
  \label{lem:sle-abstract-scaled}
  Let $\mu\in\mathbb R$ and consider the operator $A_\mu:=A+\mu I$ and the function $F_\mu\colon[\underline u,\overline u]\to E$ given by $F_\mu(v)=F(v)+\mu v$. Then $A_\mu\underline u\leq F_\mu(\underline u)$ and $A_\mu\overline u\geq F(\overline u)$ weakly. Moreover, for $u_0\in[\underline u,\overline u]$, a function $u\in C([0,\infty),E)$ is a solution of \eqref{eq:sle-abstract} if and only it is a solution of
  \begin{equation}
    \label{eq:sle-abstract-scaled}
    \begin{aligned}
      \dot u(t)+A_\mu u(t) & =F_\mu(u(t))\qquad \text{for $t>0$,} \\
      u(0)                 & =u_0.
    \end{aligned}
  \end{equation}
\end{lemma}
\begin{proof}
  Note that $D(A_\mu')=D(A')$ and that $A_\mu'=A'+\mu I$. Let $v'\in D(A')_+$. Then
  \begin{align*}
    \langle A_\mu'v',\underline u\rangle
     & =\langle A'v',\underline u\rangle + \langle\mu v',\underline u\rangle     \\
     & \leq\langle v',F(\underline u)\rangle + \langle v',\mu\underline u\rangle
    =\langle v',F_\mu(\underline u)\rangle,
  \end{align*}
  and similarly for $\overline u$. This proves the first claim. Regarding the second claim let $u$ be a solution of \eqref{eq:sle-abstract}. Then by Proposition~\ref{prop:weak-solution-characterisation} we have
  \begin{align*}
    \frac{d}{dt}\langle v',u(t)\rangle+\langle A_\mu'v',u(t)\rangle
     & =\frac{d}{dt}\langle v',u(t)\rangle+\langle A'v',u(t)\rangle+\mu\langle v',u(t)\rangle \\
     & =\langle v'F(u(t))\rangle+\mu\langle v',u(t)\rangle
    =\langle v'F_\mu(u(t))\rangle.                                                            \\
  \end{align*}
  Again using Proposition~\ref{prop:weak-solution-characterisation} it follows that $u$ is a solution of \eqref{eq:sle-abstract-scaled}. The other implication is shown similarly.
\end{proof}
Lemma~\ref{lem:sle-abstract-scaled} shows that by replacing $A$ by $A_\mu$ and $F$ by $F_\mu$ we can assume without loss of generality that $F$ is increasing in Theorem~\ref{thm:existence}. Hence we assume throughout that $F$ is increasing. In this section we prove parts (i)--(iii) of Theorem~\ref{thm:existence}.
\begin{theorem}
  \label{thm:convergence-iteration}
  Suppose that the semigroup $(S(t))_{t\geq 0}$ is compact, or that $E$ has order continuous norm. Let $u_0\in[\underline u,\overline u]$ and let $G$ be the fixed point map \eqref{eq:fixed-point-map-u0} associated with $u_0$. Let $(\underline w_n)_{n\in\mathbb N}$ and $(\overline w_n)_{n\in\mathbb N}$ be the iterated sequence from Definition~\ref{def:interated-sequences}. Then these sequences converge in $L^1([0,T],E)$ for every $T>0$. Their limits $u_{\min}$ and $u_{\max}$ are solutions of \eqref{eq:sle-abstract} and assertions (i)--(iii) of Theorem~\ref{thm:existence} hold.
\end{theorem}
\begin{proof}
  We know from Proposition~\ref{prop:iterated-sequences} that $\underline w_n\leq\underline w_{n+1}\leq\underline u$ for all $n\in\mathbb N$. If $E$ has order continuous norm, then by Lemma~\ref{lem:monotone-norm-convergence} the sequence $(\underline w_n)_{n\in\mathbb N}$ converges in $L_{\loc}^1([0,\infty),E)$. If the semigroup is compact, then $G$ is compact by Lemma~\ref{lem:G-compact}. As $\underline w_{n+1}=G(\underline w_n)$ for all $n\in\mathbb N$ the sequence $(\underline w_n)_{n\in\mathbb N}$ has a convergent sub-sequence. Since it is also monotone, Lemma~\ref{lem:monotone-norm-convergence} implies that the sequence itself converges.

  (i) By what we proved,
  \begin{equation*}
    u_{\min}:=\lim_{n\to\infty}\underline w_n
  \end{equation*}
  in the sense of $L_{\loc}^1([0,\infty),E)$. Since $G$ is continuous by Lemma~\ref{lem:G-invariant} we have $G(u_{\min})=u_{\min}$. Similar arguments show that
  \begin{equation*}
    u_{\max}:=\lim_{n\to\infty}\overline w_n
  \end{equation*}
  exists in the sense of $L_{\loc}^1([0,\infty),E)$ and $G(u_{\max})=u_{\max}$. In particular $u_{\min},u_{\max}\in C([0,\infty),E)$ are solutions of \eqref{eq:sle-abstract}. It follows from \eqref{eq:iteration-comparison} that $u_{\min}\leq u_{\max}$. Let now $u$ be a solution of \eqref{eq:sle-abstract}. Since $\underline u\leq u(t)\leq\overline u$ for all $t\geq 0$ it follows that
  \begin{equation*}
    \underline w_0
    =\underline u
    \leq G(\underline u)
    \leq G(u)
    =u
    \leq G(\overline u)
    \leq\overline u
    =\overline w_0.
  \end{equation*}
  It follows inductively that $\underline w_n\leq u\leq \overline w_n$ for all $n\in\mathbb N$. In fact, if this is true for some $n\geq 0$, then
  \begin{equation*}
    \underline w_{n+1}
    =G(\underline w_n)
    \leq G(u)
    =u\leq G(\overline w_n)
    =\overline w_{n+1}.
  \end{equation*}
  Letting $n\to\infty$ yields $u_{\min}\leq u\leq u_{\max}$. This proves part (i) of Theorem~\ref{thm:existence}. The argument also proves part (iii) of Theorem~\ref{thm:existence} by taking $u_0=\underline u$ and $u_0=\overline u$, respectively.

  (ii) To prove part (ii) of Theorem~\ref{thm:existence} we let $u_0\leq\tilde u_0$ with corresponding iterated lower sequences $(\underline w_n)_{n\in\mathbb N}$ and $(\tilde{\underline{w}}_n)_{n\in\mathbb N}$, respectively. One shows inductively that $\underline w_n\leq\tilde{\underline{w}}_n$ for all $n\in\mathbb N$ and hence $u_{\min}\leq \tilde u_{\min}$. Similarly $u_{\max}\leq\tilde u_{\max}$.
\end{proof}
The above proposition asserts the convergence of the sequences $(\underline w_n)_{n\in\mathbb N}$ and $(\overline w_n)_{n\in\mathbb N}$ in $L_{\loc}^1([0,\infty),E)$. The sequences and their limits are continuous functions, and the convergence is monotone. We can use Dini's theorem to show that the convergence is in fact locally uniform.
\begin{proposition}[Uniform convergence]
  \label{prop:uniform-convergence}
  Under the assumptions of Theorem~\ref{thm:convergence-iteration} we have that $\underline w_n\to u_{\min}$ and $\overline w_n\to u_{\max}$ in $C([0,T],E)$ for every $T>0$.
\end{proposition}
\begin{proof}
  Since the cone in $E$ is normal
  \begin{equation*}
    \tnorm{v}:=\sup_{v\in B_+'}|\langle v',v\rangle|
  \end{equation*}
  defines an equivalent norm on $E$, see Lemma~\ref{lem:normal-cone} in the appendix. Hence there exists $\alpha>0$ such that $\|v\|_E\leq\alpha\tnorm{v}$ for all $v\in E$.

  By the Banach-Alaoglu theorem $B_+'$ is compact with respect to the weak$^*$ topology. Thus the set $K:=[0,T]\times B_+'$ is compact. For $v\in C([0,T],E)$, the function given by $\tilde v(t,v'):=\langle v',v(t)\rangle$ for all $t\in[0,T]$ and $v'\in B_+'$ defines $\tilde v\in C(K)$. Due to the monotone convergence, it follows from Dini's theorem that $\tilde{\underline w}_n\to \tilde u_{\min}$ uniformly on $K$. Thus
  \begin{align*}
    \|\underline w_n-u_{\min}\|_{C([0,T],E)}
     & =\sup_{t\in[0,T]}\left\|\underline w_n(t)-u_{\min}(t)\right\|                            \\
     & \leq\alpha\sup_{t\in[0,T]}\tnorm{\underline w_n(t)-u_{\min}(t)}                          \\
     & =\alpha\sup_{(t,v')\in K}\left|\tilde{\underline w}_n(t,v')-\tilde u_{\min}(t,v')\right|
    \xrightarrow{n\to\infty}0,
  \end{align*}
  proving the uniform convergence.
\end{proof}

\section{Asymptotic behaviour: Convergence to equilibria}
\label{sec:asymptotics}
The aim of this section is to prove the last part of Theorem~\ref{thm:existence} on the convergence to an equilibrium. As before, $(S(t))_{t\geq 0}$ is a positive $C_0$ semigroup with generator $-A$. The underlying space $E$ is an ordered Banach space with normal cone, $[\underline u,\overline u]$  is an order interval in $E$, $F\colon [\underline u,\overline u]\to E$ is continuous and quasi-increasing. Moreover, $\underline u$ and $\overline u$ are sub- and super-solutions as in Definition~\ref{def:sub-equilibria}. An \emph{equilibrium} of the equation $\dot u(t)+ Au(t)=F(u(t))$ is an element $v\in D(A)$ satisfying the equation $Av=F(v)$. This is equivalent to $v$ being a stationary solution of $\dot u(t)+ Au(t)=F(u(t))$.

We are interested in the convergence of solutions of \eqref{eq:sle-abstract} as $t\to\infty$. It turns out that the limit is an equilibrium.
\begin{proposition}[Convergence to equilibria]
  \label{prop:convergence-equilibrium}
  Let $u_0\in[\underline u,\overline u]$ and let $u$ be a solution of \eqref{eq:sle-abstract}. If $\lim_{t\to\infty}u(t)=u_\infty$ exists, then $u_\infty$ is an equilibrium.
\end{proposition}
\begin{proof}
  By Lemma~\ref{lem:sle-abstract-scaled} we can choose $\omega>0$ such that
  \begin{equation*}
    \|S(t)\|\leq Me^{-\omega t}
  \end{equation*}
  for some $M\geq 1$ and for all $t\geq 0$, that is, $(S(t))_{t\geq 0}$ is exponentially stable. Then $A$ is invertible and by the Laplace transform representation of the resolvent,
  \begin{equation*}
    A^{-1}=\int_0^\infty S(t)\,dt.
  \end{equation*}
  Note that
  \begin{align*}
    u(t) & =S(t)u_0+\int_0^tS(s)F(u(t-s))\,ds                                        \\
         & =S(t)u_0+\int_0^\infty S(s)F(u(t-s))\,ds-\int_t^\infty S(s)F(u(t-s))\,ds.
  \end{align*}
  Recall the bound \eqref{eq:image-bounded-C} on $F(u(t-s))$. Due to the exponential stability of $(S(t))_{t\geq 0}$ and the boundedness of $F(u(t-s))$, the first and the last term in the above identity converge to zero as $t\to\infty$. It follows from the dominated convergence theorem and the assumption that
  \begin{equation*}
    \lim_{t\to\infty}\int_0^\infty S(s)F(u(t-s))\,ds
    =\int_0^\infty S(s)F(u_\infty)\,ds
    =A^{-1}F(u_\infty).
  \end{equation*}
  Thus, $u_\infty=A^{-1}F(u_\infty)$, that is $u_\infty\in D(A)$ and $Au_\infty=F(u_\infty)$ as claimed.
\end{proof}
We say that $u$ is a \emph{solution} of $\dot u(t)+ Au(t)=F(u(t))$ if $u\in C([0,\infty),E)$ is a solution to \eqref{eq:sle-abstract} with $u_0:=u(0)$. The semigroup $(S(t))_{t\geq 0}$ is called \emph{eventually compact} if there exists $t_0>0$ such that $S(t_0)$ is a compact operator. As a consequence, $S(t)=S(t_0)S(t-t_0)$ is compact for all $t\geq t_0$ and $S\colon [t_0,\infty)\to\mathcal L(E)$ is continuous with respect to the operator norm, see \cite[Lemma~II.4.22]{engel:00:ops}. We next look at the relative compactness of orbits.
\begin{proposition}[compact orbits]
  \label{prop:compact-orbit}
  Let $(S(t))_{t\geq 0}$ be eventually compact and let $u$ be a solution of $\dot u(t)+ Au(t)=F(u(t))$. Then the orbit $\{u(t)\colon t\geq 0\}$ is relatively compact in $E$.
\end{proposition}
\begin{proof}
  Let $t_0>0$ such that $S(t_0)$ is compact. It is clear that $\{u(t)\colon t\in[0,t_0]\}$ is compact. By Lemma~\ref{lem:sle-abstract-scaled} we can choose $\omega>0$ such that
  \begin{equation*}
    \|S(t)\|\leq Me^{-\omega t}
  \end{equation*}
  for some $M\geq 1$ and for all $t\geq 0$. If $t\geq t_0$ we have that
  \begin{equation*}
    u(t)=S(t_0)\left[S(t-t_0)u(0)+\int_0^tS(t-t_0-s)F(u(s))\,ds\right].
  \end{equation*}
  Hence, $\{u(t)\colon t\geq t_0\}$ is the image of a bounded subset of $E$ under the compact operator $S(t_0)$ and thus it is relatively compact.
\end{proof}
The following theorem proves part (iv) of Theorem~\ref{thm:existence}.
\begin{theorem}[Asymptotics]
  \label{thm:asymptotics}
  Assume that $E$ has order continuous norm or that the semigroup $(S(t))_{t\geq 0}$ is compact. Let $U_{\min}$ be the minimal solution of \eqref{eq:sle-abstract} with $u(0)=\underline u$. and let $U_{\max}$ be the maximal solution of \eqref{eq:sle-abstract} with $u(0)=\overline u$. Then the following assertions hold.
  \begin{enumerate}[\normalfont (i)]
  \item The solution $U_{\min}$ is increasing and $U_{\max}$ is decreasing. Moreover, the limits
    \begin{equation*}
      u_*=\lim_{t\to\infty}U_{\min}(t)\qquad\text{and}\qquad u^*=\lim_{t\to\infty}U_{\max}(t)
    \end{equation*}
    exist and are solutions of $Av=F(v)$.
  \item If $v\in[\underline u,\overline u]$ is a solution of $Av=F(v)$, then $u_*\leq v\leq u^*$, that is, $u_*$ and $u^*$ are the minimal and the maximal solutions of $Av=F(v)$ in $[\underline u,\overline u]$.
  \item Let $u_0\in [\underline u,u_*]$ and let $u_{\min}$ be the minimal solution of \eqref{eq:sle-abstract}. Then
    \begin{equation*}
      \lim_{t\to\infty}u_{\min}(t)=u_*
    \end{equation*}
    and similarly if $u_0\in [u^*,\overline u]$, then
    \begin{equation*}
      \lim_{t\to\infty}u_{\max}(t)=u^*.
    \end{equation*}
  \end{enumerate}
\end{theorem}
\begin{proof}
  (i) To prove that $U_{\min}\colon [0,\infty)\to E$ is increasing we use the autonomous nature of the problem. Given $t_0\geq 0$ and $t>0$ we let
  \begin{displaymath}
    u(t):=U_{\min}(t_0+t).
  \end{displaymath}
  Then by Lemma~\ref{lem:shift} the function $u$ solves equation \eqref{eq:sle-abstract} with initial condition $u(0)=U_{\min}(t_0)$. Since $U_{\min}(t_0)\geq\underline u$, the comparison principle from Theorem~\ref{thm:existence}(iii) proved in Theorem~\ref{thm:convergence-iteration} implies that
  \begin{displaymath}
    U_{\min}(t)\leq u(t)=U_{\min}(t_0+t)
  \end{displaymath}
  for all $t_0,t\geq 0$. Hence $U_{\min}$ is increasing as a function of $t$. A similar argument shows that $U_{\max}$ is decreasing as a function of $t$.

  Next consider the asymptotic limits of these solutions. If $E$ has order continuous norm, then the limits $u_*$ and $u^*$ exist. If $(S(t))_{t\geq 0}$ is compact, then by Proposition~\ref{prop:compact-orbit} the orbit $\{U_{\min}(t)\colon t\geq 0\}$ is relatively compact. It follows from Lemma~\ref{lem:monotone-norm-convergence} that $u_*:=\lim_{t\to\infty}U_{\min}(t)$ exists, and similarly for $u^*$. They are solutions of $Av=F(v)$ by Proposition~\ref{prop:convergence-equilibrium}.

  (ii) We need to show that $u_*$ and $u^*$ are the minimal and the maximal solutions of $Av=F(v)$ in $[\underline u,\overline u]$. Hence let $v\in D(A)\cap[\underline u,\overline u]$ satisfy $Av=F(v)$. In particular, $v$ is a solution of \eqref{eq:sle-abstract} with initial condition $u_0=v$. It follows from Theorem~\ref{thm:convergence-iteration} that $U_{min}(t)\leq v\leq U_{max}(t)$ for all $t\geq 0$. As the cone is closed it follows that
  \begin{equation*}
    u_*=\lim_{t\to\infty}U_{\min}(t)\leq v\leq\lim_{t\to\infty}U_{\max}(t)=u^*,
  \end{equation*}
  proving our claim.

  (iii) Let $u_0\in [\underline u,u_*]$ and let $u_{\min}$ be the corresponding minimal solution of \eqref{eq:sle-abstract}. Let $\tilde u_{\min}$ be the minimal solution with initial condition $u_*$. Since $u_*$ is a solution with initial condition $u_*$ it follows from Theorem~\ref{thm:convergence-iteration} that
  \begin{equation*}
    U_{\min}(t)\leq u_{\min}(t)\leq \tilde u_{\min}(t)\leq u_*
  \end{equation*}
  and thus
  \begin{equation*}
    0\leq u_*-u_{\min}(t)\leq u_*-U_{\min}(t)
  \end{equation*}
  for all $t\geq 0$. Since $E$ has a normal cone, by Lemma~\ref{lem:normal-cone} there exists a monotone equivalent norm $\tnorm{{\cdot}}$ on $E$. Hence, by (i)
  \begin{equation*}
    \lim_{t\to\infty}\tnorm{u_*-u_{\min}(t)}
    \leq\lim_{t\to\infty}\tnorm{u_*-U_{\min}(t)}=0,
  \end{equation*}
  that is, $u_{\min}(t)\to u_*$ as $t\to\infty$.
\end{proof}
As a direct consequence we obtain the following corollary.
\begin{corollary}
  Assume that $E$ has order continuous norm or that the semigroup $(S(t))_{t\geq 0}$ is compact. Then, $Av=F(v)$ has an equilibrium in $[\underline u,\overline u]$. Moreover, if there is exactly one equilibrium $v$ in $[\underline u,\overline u]$, then $u(t)\to v$ as $t\to\infty$ for every solution of $\dot u(t)+Au(t)=F(u(t))$ with initial value in $[\underline u,\overline u]$.
\end{corollary}
\begin{proof}
  We observe that by Theorem~\ref{thm:asymptotics} the equation $Av=F(v)$ has the equilibria $u_*$ and $u^*$ In case of a unique equilibrium we must have $v=u_*=u^*$. Then Theorem~\ref{thm:asymptotics}(iii) implies that every orbit converges $v$.
\end{proof}

\begin{remark}
  Under the assumptions of the above theorem, the maps $[\underline u,\overline u]\to C([0,\infty),E)$ given by $u_0\mapsto u_{\min}$ or $u_0\mapsto u_{\max}$ are \emph{monotone dynamical systems} as defined in \cite{hirsch:88:scs}. We prove in Theorem~\ref{thm:asymptotics} that certain orbits converge to an equilibrium. One could ask whether all relatively compact orbits converge to some equilibrium. This is not true in general as \cite[page~2]{hirsch:88:scs} shows. It is possible to have time-periodic orbits.
\end{remark}

\section{Uniqueness of solutions}
\label{sec:uniqueness}
Our theory gives for each initial condition a maximal and a minimal solution to \eqref{eq:sle-abstract}. The following uniqueness result has a quite standard proof. Let $(S(t))_{t\geq 0}$ be a $C_0$-semigroup on a Banach space $X$ and let $K\subseteq X$ be closed and bounded and $F\colon K\to X$ \emph{locally Lipschitz continuous}, that is, for each $v\in K$ there exist $L,\varepsilon>0$ such that
\begin{equation}
  \label{eq:loc-Lip}
  \|F(v_2)-F(v_1)\|\leq L\|v_2-v_1\|
\end{equation}
for all $v_1,v_2\in B(v,\varepsilon)\cap K$. Let $\tau>0$. Given $u_0\in K$, a mild solution of
\begin{equation}
  \label{eq:evp-general}
  \begin{aligned}
    \dot u+Au&=F(u)&&\text{for }t\in[0,\tau]\\
    u(0)&=u_0&&
  \end{aligned}
\end{equation}
is a function $u\in C([0,\tau],K)$ such that
\begin{displaymath}
  u(t)=u_0+\int_0^sS(t-s)F(u(s))\,ds
\end{displaymath}
for all $t\in[0,\tau]$.

\begin{proposition}[Uniqueness of solutions]
  \label{prop:uniqueness}
  Given $u_0\in K$ there exists at most one mild solution of \eqref{eq:evp-general}.
\end{proposition}
\begin{proof}
  Let $u_1,u_2\in C([0,\tau],E)$ be solutions of \eqref{eq:sle-abstract}. Define
  \begin{equation*}
    J:=\{t\in[0,\tau]\colon u_1(s)=u_2(s)\text{ for all }s\in[0,t]\}.
  \end{equation*}
  We show that $J$ non-empty, closed and open in $[0,\tau]$ and hence $J=[0,\tau]$. We first note that $0\in J$, so $J\neq\emptyset$.  We show that $J$ is closed. Assume that $t_n\in J$ with $t_n\uparrow t_0$ as $n\to\infty$. By assumption $u_1(t_n)=u_2(t_n)$ for all $n\in\mathbb N$ and thus by the continuity of $u_1$ and $u_2$ we have $u_1(t_0)=u_2(t_0)$. Hence $t_0\in J$. If $t_n\downarrow t_0$, then $t_0\in J$ by definition of $J$. Next we show that $J$ is open. We let $t_0\in J$. In particular we note that $u_1(t)=u_2(t)$ for all $t\in [0,t_0]$. As $F$ is locally Lipschitz there exist $L,\varepsilon>0$ such that \eqref{eq:loc-Lip} holds for all $v_1,v_2\in[\underline u,\overline u]\cap B(u(t_0),\varepsilon)$. By the continuity of $u_1$ and $u_2$ there exists $\delta>0$ such that $\|u_k(t_0)-u_k(s)\|<\varepsilon$ whenever $s\in[t_0,t_0+\delta)$ and $k=1,2$. By the Lipschitz condition and the fact that $u_1(s)=u_2(s)$ for all $s\in[0,t_0]$ this implies that
  \begin{align*}
    \|u_2(t)-u_1(t)\| & \leq\int_0^t\|S(t-s)\|\left\|F(u_2(s))-F(u_1(s))\right\|\,ds \\
                      & \leq L\int_0^t\|S(t-s)\|\|u_2(s)-u_1(s)\|\,ds                \\
                      & \leq LM\int_0^t\|u_2(s)-u_1(s)\|\,ds
  \end{align*}
  for all $t\in[0,t_0+\delta)$, where $M:=\sup\left\{\|S(s)\|\colon s\in[0,t_0+\delta]\right\}$. Now Gronwall's inequality implies that $\|u_1(t)-u_2(t)\|=0$ for all $t\in[0,t_0+\delta)$. This proves that $J$ is open. By the connectedness of $[0,\tau]$ it follows that $u_1(t)=u_2(t)$ for all $t\in[0,\tau]$.
\end{proof}

We next show that non-uniqueness is possible, even in the scalar case.

\begin{example}
  Consider the differential equation
  \begin{equation}
    \label{eq:scalar-non-unique}
    \begin{aligned}
      \dot u(t)+au(t) & =F(u(t))\qquad t\geq 0 \\
      u(0)            & =u_0
    \end{aligned}
  \end{equation}
  with $a>0$ and
  \begin{equation*}
    F(\xi):=\sign(\xi)\sqrt{|\xi|}
  \end{equation*}
  Then $F\colon\mathbb R\to\mathbb R$ is strictly increasing. Moreover, setting $Av:=av$ for all $v\in\mathbb R$ we have
  \begin{equation*}
    A(-M)\leq-\sqrt{|-M|}\qquad\text{and}\qquad AM\geq\sqrt{1}=F(M).
  \end{equation*}
  whenever $M\geq a^{-2}$. Hence, setting $\underline u:=-M$ and $\overline u:=M$ with $M\geq a^{-2}$ the initial value problem \eqref{eq:scalar-non-unique} fits into the framework of Theorem~\ref{thm:existence}. As a consequence, given any initial condition $u_0\in\mathbb R$ we choose $M\geq\max\{|u_0|,a^{-2}\}$ so that $u_0\in [-M,M]$. Then Theorem~\ref{thm:existence} implies the existence of a minimal and a maximal solution of \eqref{eq:scalar-non-unique}. Here we do not make use of any of the standard existence theorems. We also note that since $F$ is an odd function, if $u(t)$ is a solution, then also $-u(t)$ is a solution.

  By computing the solutions we show that the minimal and the maximal solutions with initial value $u_0=0$ are not the same, and that there are many solutions in between. We can solve the differential equation by separation of variables. Doing so for $u>0$, we obtain
  \begin{equation*}
    \int\frac{du}{\sqrt{u}(1-a\sqrt{u})}=\int dt=t-t_0.
  \end{equation*}
  For the integral on the left hand side we make the substitution $v=\sqrt{u}$. Then $du=2v\,dv$ and thus
  \begin{equation}
    \label{eq:variable-separation}
    \begin{split}
      \int\frac{du}{\sqrt{u}(1-a\sqrt{u})}
      & =\int\frac{2vdv}{v(1-av)}\,dv
      =2\int\frac{dv}{1-av}\,dv        \\
      & =-\frac{2}{a}\log|1-av|
      =-\frac{2}{a}\log|1-a\sqrt{u}|=t-t_0.
    \end{split}
  \end{equation}
  The expression makes sense for $u=0$ and thus a solution with $u_0=0$ is not unique. In that case we also have $1-a\sqrt{u}>0$. Solving for $u$ yields
  \begin{equation}
    \label{eq:solution-non-unique-general}
    u(t)=\frac{1}{a^2}\left(1-e^{-\frac{a}{2}(t-t_0)}\right)^2.
  \end{equation}
  Since $F$ is an odd function, the solutions of \eqref{eq:scalar-non-unique} with $u_0=0$ are given by
  \begin{equation}
    \label{eq:solution-non-unique-0}
    u(t):=
    \begin{cases}
      \pm\dfrac{1}{a^2}\left(1-e^{-\frac{a}{2}(t-t_0)}\right)^2 & \text{if }t\geq t_0   \\
      0                                                         & \text{if }0\leq t<t_0
    \end{cases}
  \end{equation}
  for any $t_0\geq 0$. In particular,
  \begin{equation*}
    u_{\max}(t)=\dfrac{1}{a^2}\left(1-e^{-\frac{a}{2}t}\right)^2
    \quad\text{and}\quad
    u_{\min}(t)=-\dfrac{1}{a^2}\left(1-e^{-\frac{a}{2}t}\right)^2
  \end{equation*}
  for $t\geq 0$. For $t_0>0$, the solutions given by \eqref{eq:solution-non-unique-0} are between the two, see Figure~\ref{fig:non-unique}.

  Taking $\underline u=M>a^{-2}$ we can also compute $U_{\max}$. Solving \eqref{eq:variable-separation} for $u$ in case of $1-a\sqrt{u}<0$ we obtain
  \begin{equation*}
    u(t)=\frac{1}{a^2}\left(1+e^{-\frac{a}{2}(t-t_0)}\right)^2.
  \end{equation*}
  for some $t_0\in\mathbb R$. Setting $u=M$ in \eqref{eq:variable-separation} we see that if we set
  \begin{equation*}
    t_M:=\frac{a}{2}\log\left(a\sqrt{M}-1\right),
  \end{equation*}
  then
  \begin{equation*}
    U_{\max}(t)=\frac{1}{a^2}\left(1+e^{-\frac{a}{2}(t-t_M)}\right)^2
    \quad\text{and}\quad
    U_{\min}(t)=-\frac{1}{a^2}\left(1+e^{-\frac{a}{2}(t-t_M)}\right)^2.
  \end{equation*}
  As expected by Theorem~\ref{thm:existence} these solutions are monotone and they converge to the equilibria $u^*=a^{-2}$ and $u_*=-a^{-2}$. The third equilibrium is $u=0$. Figure~\ref{fig:non-unique} shows the equilibria and some solutions on $\orderint{-M,M}$.

  \begin{figure}[ht]
    \label{fig:non-unique}
    \centering
    \begin{tikzpicture}[scale=1.5,samples=100,
        declare function={
            a=1;%
            M=1.7;%
            C = 2*ln(a*sqrt(M)-1)/a;%
            T=7.5;%
            u(\t,\s)=pow(1-exp(-a*(\t-\s)/2),2)/pow(a,2);%
            v(\t,\s)=pow(1+exp(-a*(\t-\s)/2),2)/pow(a,2);%
          }
      ]
      \draw[->] (-1,0) -- (T+0.5,0) node[below] {$t$};%
      \draw[->] (0,-M-0.5) -- (0,M+0.5) node[left] {$u$};%
      \begin{scope}[blue]
        \foreach \s in {1.5,3,...,6}{%
            \draw (0,0) -- plot[domain=\s:T] (\x,{u(\x,\s)});%
            \draw (0,0) -- plot[domain=\s:T] (\x,{-u(\x,\s)});%
            \draw plot[domain=\s:T] (\x,{v(\x,C+\s)});%
            \draw plot[domain=\s:T] (\x,{-v(\x,C+\s)});%
          }%
        \foreach \s in {-3.5,-1.5}{%
            \draw plot[domain=0:T] (\x,{u(\x,\s)});%
            \draw plot[domain=0:T] (\x,{-u(\x,\s)});%
            \draw plot[domain=0:T] (\x,{v(\x,C+\s)});%
            \draw plot[domain=0:T] (\x,{-v(\x,C+\s)});%
          }%
      \end{scope}
      \draw[thick] (T,M) -- (0,M) node[left] {$M$};%
      \draw[thick] (T,-M) -- (0,-M) node[left] {$-M$};%
      \node at (1.5,{u(1.5,0)}) [above] {$u_{\max}$};%
      \node at (1.5,{-u(1.5,0)}) [below] {$u_{\min}$};%
      \node at (1.5,{v(1.5,C)}) [above] {$U_{\max}$};%
      \node at (1.5,{-v(1.5,C)}) [below] {$U_{\min}$};%
      \begin{scope}[red,thick]
        \draw plot[domain=0:T] (\x,{u(\x,0)});%
        \draw plot[domain=0:T] (\x,{-u(\x,0)});%
        \draw plot[domain=0:T] (\x,{v(\x,C)});%
        \draw plot[domain=0:T] (\x,{-v(\x,C)});%
      \end{scope}
      \draw[thick,dashed] (T,1/a^2) -- (0,1/a^2) node[left] {$1/a^2$};%
      \draw[thick,dashed] (T,-1/a^2) -- (0,-1/a^2) node[left] {$-1/a^2$};%
      \draw[thick,dashed] (T,0) -- (0,0) node[below left] {$0$};%
    \end{tikzpicture}
    \caption{Maximal and minimal solutions (shown in red) of~\eqref{eq:scalar-non-unique} on $\orderint{-M,M}$.}
  \end{figure}
\end{example}

\section{Admissible Operators}
\label{sec:operators}
In this section we give several examples to show how the results can be applied. As before $\Omega\subseteq\mathbb R^N$ is a non-empty, bounded, connected open set. We will consider three ordered Banach spaces, namely
\begin{itemize}
\item $E=L^p(\Omega)$, $1\leq p<\infty$,
\item $E=C(\bar\Omega)$
\item $C_0(\Omega):=\{u\in C(\bar\Omega)\colon u|_{\partial\Omega}=0\}$
\end{itemize}
with their natural norms. Let $A$ be an operator on $E$. In the following definition we collect properties we will use in the applications given in Section~\ref{sec:applications}. Frequently they are stronger than needed, but they make the arguments less technical than minimal assumptions.
\begin{definition}[Admissible operator]
  \label{def:A-admissible}
  An operator $A$ on $E$ is \emph{admissible} if $-A$ generates a positive, irreducible, sub-markovian $C_0$-semigroup $(S(t))_{t\geq 0}$ on $E$ and the additional following properties hold:
  \begin{enumerate}[(a)]
  \item $S(t)E\subseteq L^\infty(\Omega)$ for all $t>0$ if $E=L^p(\Omega)$, $1\leq p<\infty$;\label{item:ultracontractive}
  \item $S(t)$ is compact for all $t>0$ if $E=C(\bar\Omega)$ or $E=C_0(\Omega)$
  \end{enumerate}
\end{definition}
We comment on the diverse properties. At first we discuss the irreducibility, which is of different nature in $C(\bar\Omega)$, $C_0(\Omega)$ and $L^p(\Omega)$. If $v\in E$ we write $v>0$ if $v\geq 0$ with $v\neq 0$. We write
\begin{itemize}
\item $v\gg 0$ if $v(x)>0$ almost everywhere if $E=L^p(\Omega)$;
\item $v\gg 0$ if $v(x)>0$ for all $x\in\bar\Omega$ if $E=C(\bar\Omega)$;
\item $v\gg 0$ if $v(x)>0$ for all $x\in\Omega$ if $E=C_0(\Omega)$.
\end{itemize}
Note that in the second case, $v(x)\geq\delta>0$ for all $x\in\bar\Omega$ and some $\delta>0$. Let $(S(t))_{t\geq 0}$ be a $C_0$-semigroup on $E$ with generator $-A$. Then there exists $\lambda_0\in\mathbb R$ such that $R(\lambda,A):=(\lambda I+A)^{-1}$
exists for all $\lambda>\lambda_0$. The semigroup $(S(t))_{t\geq 0}$ is positive if and only if $R(\lambda,A)v\geq 0$ whenever $0\leq v\in E$ and $\lambda>\lambda_0$. Moreover, we call $(S(t))_{t\geq 0}$ \emph{irreducible} if it is positive and $R(\lambda,A)v\gg 0$ for all $\lambda>\lambda_0$ and $0<v\in E$. In case (\ref{item:ultracontractive}) this implies that $S(t)$ is compact for all $t>0$. In the case $E=C_0(\Omega)$ and $C(\bar\Omega)$ we incorporate this into the definition.

A linear mapping $T\colon E\to E$ is called \emph{sub-markovian} if
\begin{displaymath}
  v\leq 1_\Omega\implies Tv\leq 1_\Omega
\end{displaymath}
for all $v\in E$. This is equivalent to saying that $T\geq 0$ and $\|Tv\|_{L^\infty}\leq \|v\|_{L^\infty}$ for all $v\in E$. A semigroup $(S(t))_{t\geq 0}$ on $E$ is called \emph{sub-markovian} if $S(t)$ is sub-markovian for all $t>0$. Let $(S(t))_{t\geq 0}$ be a positive $C_0$-semigroup on $L_p(\Omega)$, $1\leq p<\infty$ or $C(\bar\Omega)$. Then  $(S(t))_{t\geq 0}$ is sub-markovian if and only if
\begin{equation}
  \label{eq:A-submarkov}
  \langle 1,A'v\rangle\geq 0
\end{equation}
for all $0<v'\in D(A')$, where $-A$ is the generator of $(S(t))_{t\geq 0}$. Indeed, since $(S(t))_{t\geq 0}$ is positive, it is sub-markovian if and only if $S(t)1_\Omega\leq 1_\Omega$ for all $t\geq 0$. If $0<v\in D(A')$ this implies that
\begin{displaymath}
  \langle 1_\Omega,A'v\rangle = \lim_{t\to 0+}\left\langle\frac{1_\Omega-S(t)1_\Omega}{t},v\right\rangle\geq 0.
\end{displaymath}
Conversely, if \eqref{eq:A-submarkov} holds, then for $0\leq v'\in D(A')$
\begin{displaymath}
  \langle S(t)1_\Omega,v'\rangle-\langle 1_\Omega,v'\rangle
  =-\int_0^t\langle 1_{\Omega}, A'S(s)'v'\rangle\,ds
  \leq 0,
\end{displaymath}
implying that $S(t)1_\Omega\leq 1_\Omega$. We will use \eqref{eq:A-submarkov} to show that constant functions are super-solutions of some stationary problems we will consider. Other more sophisticated constructions for abstract operators appear in \cite{arendt:23:see} and in more concrete cases for instance in \cite{daners:18:gdg,fraile:eep,du:06:ost} and others.

We give some consequences of our assumptions which will be required below.

\begin{theorem}
  \label{thm:pev}
  Let $A$ be an admissible operator. Then there exists a unique $\lambda\in\mathbb R$ such that there exists $0<\varphi\in D(A)$ with
  \begin{displaymath}
    A\varphi=\lambda\varphi.
  \end{displaymath}
  In that case $\varphi$ is bounded and $\varphi\gg 0$. Moreover, $\varphi$ is unique if we require in addition that $\|\varphi\|_E=1$ and $\lambda$ is the smallest eigenvalue of $A$.
\end{theorem}

For a proof of the above theorem we refer for instance to \cite{nagel:86:osp}. We denote the unique eigenvalue in the above theorem by $\lambda_1(A)$ and call it the \emph{principal eigenvalue} of $A$. The corresponding positive eigenvector $u$ with $\|u\|=1$ is called the \emph{principal eigenvector} of $A$.

Note that $S(t)\varphi=e^{-\lambda_1t}\varphi$ for all $t>0$. Thus in the case where $E=L^p(\Omega)$, property (\ref{item:ultracontractive}) in Definition~\ref{def:A-admissible} implies that $\varphi\in L^\infty(\Omega)$.

For many examples we will make use of the following fact that follows with a proof very similar to that given in \cite[Theorem~3.1]{arendt:23:see} for the $L^p$-spaces.

\begin{theorem}[Strict monotonicity of principal eigenvalue]
  \label{thm:ev-monotone}
  Let $m\in L^\infty(\Omega)$ in the case $E=L^p(\Omega)$, and $m\in C(\bar\Omega)$ in the case $E=C(\bar\Omega)$ and $m\in BC(\Omega)$ in the case $E=C_0(\Omega)$. If $A$ is an admissible operator, then $A+m$ is also admissible. Moreover for $m_1\leq m_2$,
  \begin{displaymath}
    \lambda_1(A+m_1)\leq \lambda_1(A+m_2)
  \end{displaymath}
  with equality if and only if $m_1=m_2$ in $E$.
\end{theorem}
We now give several examples of admissible operators.

\begin{example}[Dirichlet Laplacian]
  \label{ex:A-dirichlet-laplacian}
  (a) On $L^2(\Omega)$ the Dirichlet Laplacian can be defined without further assumptions on $\Omega$. Define
  \begin{displaymath}
    \begin{aligned}
      D(A)&:=\{u\in H_0^1(\Omega)\colon \Delta u\in L^2(\Omega)\}\\
      Au&:=-\Delta u\qquad\text{for }u\in D(A).
    \end{aligned}
  \end{displaymath}
  Then $A$ is admissible. If $\Omega$ is convex or has $C^2$-boundary, then $D(A)=H_0^1(\Omega)\cap H^2(\Omega)$.

  (b) For each $p\in [1,\infty)$, there exists an admissible operator $A_p$ on $L^p(\Omega)$ such that the semigroups $(S_p(t))_{t\geq 0}$ generated by $-A_p$ are consistent, that is, $S_p(t)u=S_q(t)u$ for all $u\in L^p(\Omega)\cap L^q(\Omega)$, $1\leq p,q<\infty$. Moreover, $A_2$ is the operator introduced in (a).

  (c) To define the Dirichlet Laplacian on $C_0(\Omega)$ assume that $\Omega$ is Wiener regular, which is for instance the case if $\Omega$ has Lipschitz boundary and $N\geq 2$. If $N=2$, then it is sufficient that $\Omega$ is simply connected. Define the operator by
  \begin{displaymath}
    \begin{aligned}
      D(A)&:=\{u\in C_0(\Omega)\colon \Delta u\in C_0(\Omega)\}\\
      Au&:=-\Delta u\qquad\text{for }u\in D(A).
    \end{aligned}
  \end{displaymath}
  Then $A$ is admissible, see \cite[Theorem~2.3, Proposition~3.2]{arendt:99:wrh} or \cite[Theorem~6.1.8 and Theorem~6.3.1]{arendt:11:vlt}.
\end{example}

For the next example we need the weak normal derivative.

\begin{definition}[Weak normal derivative]
  \label{def:wnd}
  Assume that $\Omega$ has Lipschitz boundary. Consider $\partial\Omega$ with the surface measure $\sigma$. Then there exists a unique bounded operator
  \begin{displaymath}
    \trace\colon H^1(\Omega)\to L^2(\partial\Omega)
  \end{displaymath}
  such that $\trace(u)=u|_{\partial\Omega}$ whenever $u\in H^1(\Omega)\cap C(\bar\Omega)$. Let $u\in H^1(\Omega)$ such that $\Delta u\in L^2(\Omega)$.

  \begin{enumerate}[(a)]
  \item Let $g\in L^2(\partial\Omega)$. We define the \emph{normal
      derivative} by Green's formula:
    \begin{displaymath}
      \partial_\nu u=g:\iff\int_\Omega\nabla u\nabla v\,dx+\int_\Omega v\Delta u\,dx =\int_{\partial\Omega}gv\,d\sigma
    \end{displaymath}
    for all $v\in C^1(\bar\Omega)$.
  \item We say that $\partial_\nu u\in L^2(\partial\Omega)$ if there exists $g\in L^2(\partial\Omega)$ such that $\partial_\nu u=g$.
  \item We say $\partial_\nu u\in C(\partial\Omega)$ if there exists $g\in C(\partial\Omega)$ such that $\partial_\nu u=g$.
  \end{enumerate}
\end{definition}

\begin{example}[Neumann Laplacian]
  \label{ex:A-neumann-laplacian}
  Assume that $\Omega$ has Lipschitz boundary.

  (a) Let $E=L^2(\Omega)$. Define $A$ on $L^2(\Omega)$ by
  \begin{displaymath}
    \begin{aligned}
      D(A)&:=\{u\in H^1(\Omega)\colon \Delta u\in L^2(\Omega),\partial_\nu u=0\}\\
      Au&:=-\Delta u\qquad\text{for }u\in D(A).
    \end{aligned}
  \end{displaymath}
  Then $A$ is admissible.

  (b) There exist consistent semigroups $(S_p(t))_{t\geq 0}$ on $L^p(\Omega)$, $1\leq p<\infty$, such that their generators $-A_p$ are admissible and $A_2=A$ from (a).

  (c) Let $E=C(\bar\Omega)$. Define $A$ on $C(\bar\Omega)$ by
  \begin{displaymath}
    \begin{aligned}
      D(A)&:=\{u\in H_0^1(\Omega)\cap C(\bar\Omega)\colon \Delta u\in C(\bar\Omega), \partial_\nu u=0\}\\
      Au&:=-\Delta u\qquad\text{for }u\in D(A).
    \end{aligned}
  \end{displaymath}
  Then $A$ is admissible. The semigroup $(S(t))_{t\geq 0}$ generated by $-A$ is the restriction of $(S_p(t))_{t\geq 0}$ to $C(\bar\Omega)$, $1\leq p<\infty$. The irreducibility of the semigroup on $C(\bar\Omega)$ is not obvious, we refer to \cite[Corollary~3.2]{arendt:20:spp}.
\end{example}

\begin{example}[Robin Laplacian]
  \label{ex:A-robin-laplacian}
  Assume that $\Omega$ has Lipschitz boundary. Let $0\leq\beta\in L^\infty(\partial\Omega)$.

  (a) Define $A$ on $L^2(\Omega)$ by
  \begin{displaymath}
    \begin{aligned}
      D(A)&:=\{u\in H^1(\Omega)\colon \Delta u\in L^2(\Omega),\partial_\nu u=-\beta\trace(u)\}\\
      Au&:=-\Delta u\qquad\text{for }u\in D(A).
    \end{aligned}
  \end{displaymath}
  Then $A$ is admissible, see \cite[Theorem~8.3]{arendt:23:see}.

  (b) Define $A$ on $C(\bar\Omega)$ by
  \begin{displaymath}
    \begin{aligned}
      D(A)&:=\{u\in H_0^1(\Omega)\cap C(\bar\Omega)\colon \Delta u\in C(\bar\Omega), \partial_\nu u=-\beta\trace(u)\}\\
      Au&:=-\Delta u\qquad\text{for }u\in D(A).
    \end{aligned}
  \end{displaymath}
  Then $A$ is admissible on $C(\bar\Omega)$. See \cite[Theorem~4.3]{nittka:11:rsl}, who shows that $-A$ generates a positive, holomorphic $C_0$-semigroup. Irreducibility in $C(\bar\Omega)$ is much stronger than irreducibility in $L^2(\Omega)$ and follows from \cite[Corollary~3.2]{arendt:20:spp}.
\end{example}

\begin{example}[Ellipic operators in non-divergence form]
  \label{ex:A-non-divergence}
  Assume that $\Omega\subseteq\mathbb R^N$ satisfies a uniform exterior cone condition as in \cite[page~203]{gilbarg:01:epd}. Let $a_{jk}=a_{kj}\in C(\bar\Omega)$, $1\leq j,k\leq N$, such that for some $\alpha>0$
  \begin{displaymath}
    \sum_{k=1}^N\sum_{j=1}^Na_{jk}\xi_j\xi_k>\alpha|\xi|^2
  \end{displaymath}
  for all $x\in\bar\Omega$, $\xi\in\mathbb R^N$, and let $c,b_j\in L^\infty(\Omega)$, $j=1,\dots,N$, such that $c\leq 0$. Define $\mathcal A\colon W_{\loc}^{2,N}(\Omega)\to L^2(\Omega)$ by
  \begin{displaymath}
    \mathcal Av:=\sum_{k=1}^N\sum_{j=1}^Na_{jk}\partial_j\partial_j v+\sum_{j=1}^Na_{j}\partial_jv+cv.
  \end{displaymath}
  Define the operator $A$ on $C_0(\Omega)$ by
  \begin{displaymath}
    \begin{aligned}
      D(A)&:=\{u\in C_0(\Omega)\cap W_{\loc}^{2,N}(\Omega)\colon \mathcal Au\in C_0(\Omega)\}\\
      Au&:=-\mathcal Au\qquad\text{for }u\in D(A).
    \end{aligned}
  \end{displaymath}
  for all $u\in D(A)$. Then $A$ is admissible by \cite[Theorem~3.1, Proposition~3.4 and Proposition~3.8]{arendt:14:sge}.
\end{example}

\section{Applications}
\label{sec:applications}
In this section we consider three semi-linear equations which we treat for the elliptic operators considered in Section~\ref{sec:operators}. The aim is to illustrate how the theory from previous sections applies in simple situations with minimal assumptions, not emphasising the most general conditions on the non-linearities.

\subsection{The logistic equation}
\label{ex:logistic-equation}
Let $A$ be an admissible operator, where $E$ is one of the spaces $L^p(\Omega)$, $1\leq p\leq\infty$, $C(\bar\Omega)$ or $C_0(\Omega)$. Let $a>0$, $b>0$ be constants. We study the logistic equation
\begin{equation}
  \label{eq:parabolic-logistic}
  \dot u(t)+Au(t)=au(t)-bu(t)^2\quad\text{for }t>0
\end{equation}
Recall that an equilibrium of \eqref{eq:parabolic-logistic} is a function $0<u_\infty\in D(A)$ such that
\begin{equation}
  \label{eq:stationary-logistic}
  Au_\infty=au_\infty-bu_\infty^2.
\end{equation}
We denote by $\varphi_0$ the principal eigenvector of $A$. Note that $\varphi_0$ is bounded and $\varphi_0\gg 0$ in $E$.
\begin{proposition}[Existence of equilibria]
  \label{prop:logistic-existence}
  If \eqref{eq:parabolic-logistic} has an equilibrium $u_\infty$, then $\lambda_1(A)<a$. Moreover, there exists at most one equilibrium
\end{proposition}
\begin{proof}
  Assume that $0<u_\infty\in D(A)$ is an equilibrium of \eqref{eq:parabolic-logistic}. Then
  \begin{displaymath}
    (A+bu_\infty)u_\infty=au_\infty,
  \end{displaymath}
  so $u_\infty$ is a positive eigenvector of the operator $A+bu_\infty$. It follows from Theorem~\ref{thm:pev} that $\lambda_1(A+bu_\infty)=a$. Since $bu_\infty>0$, by Theorem~\ref{thm:ev-monotone} we have
  \begin{displaymath}
    \lambda_1(A)<\lambda_1(A+bu_\infty)u_\infty=a
  \end{displaymath}
  as claimed. To prove the uniqueness of the equilibrium assume that $0<u,v\in D(A)$ are both equilibria of \eqref{eq:parabolic-logistic}. In particular $(Au+bu)u=au$, so $\lambda_1(A+bu)=a$ by Theorem~\ref{thm:pev}. Assume that $w:=u-v\neq 0$. Then $Aw+b(u+v)w=aw$. Thus $a$ is an eigenvalue of $A+b(u+v)$, which generates a positive, irreducible $C_0$-semigroup on $E$. Consequently $\lambda_1(A+b(u+v))\leq a$. Since, by Theorem~\ref{thm:ev-monotone},
  \begin{displaymath}
    a=\lambda_1(A+bu)<\lambda_1(A+b(u+v))\leq a
  \end{displaymath}
  we obtain a contradiction. Hence, $w=u-v=0$, so $u=v$.
\end{proof}
In the next theorem we have to exclude $C_0(\Omega)$ since $1_\Omega\not\in C_0(\Omega)$.
\begin{theorem}
  \label{thm:logistic-existence}
  Let $E=L^p(\Omega)$, $1\leq p<\infty$ or $E=C(\bar\Omega)$. Assume that $\lambda_1(A)<a$. Then the following statements hold:
  \begin{enumerate}[\upshape(a)]
  \item The equation \eqref{eq:parabolic-logistic} has a unique equilibrium $u_\infty$.
  \item For each bounded $0\leq u_0\in E$ there exists a unique mild solution $u$ of \eqref{eq:parabolic-logistic} with initial value $u(0)=u_0$.
  \item If in addition $u_0\geq\varepsilon\varphi_0$ for some $\varepsilon>0$, then $\lim_{t\to\infty}u(t)=u_\infty$ in $E$, where $\varphi_0$ is the principal eigenvector of $A$.
  \end{enumerate}
\end{theorem}
\begin{proof}
  (a) Define $F\colon E_+\cap L^\infty(\Omega)\to E$ by $F(u):=au-bu^2$. We show that $\varepsilon\varphi_0$ is a sub-solution if $\varepsilon>0$ is small enough. In fact,
  \begin{displaymath}
    A(\varepsilon\varphi_0)=a\varepsilon\varphi_0-b(\varepsilon\varphi_0)^2-[a-\lambda_1(A)-b\varepsilon\varphi_0]\varepsilon\varphi_0.
  \end{displaymath}
  Since $a>\lambda_1(A)$ by Proposition~\ref{prop:logistic-existence} and since $\varphi_0$ is bounded, there exists $\varepsilon_0>0$ such that
  \begin{displaymath}
    a-\lambda_1(A)-b\varepsilon_0\varphi_0\geq 0.
  \end{displaymath}
  Hence $\underline u:=\varepsilon\varphi_0$ is a sub-solution of \eqref{eq:stationary-logistic} for all $\varepsilon\in(0,\varepsilon_0]$. Let $M>0$ be such that $Mb>a$ and $M1_\Omega\geq \varepsilon_0\varphi_0$. Then $\overline u:=M1_\Omega$ is a super-solution. In fact,
  \begin{displaymath}
    \langle v',F(M1_\Omega)\rangle\leq 0\leq\langle A'v',M1_\Omega\rangle
  \end{displaymath}
  for all $0\leq v'\leq D(A')$.

  We show that $F$ is quasi-increasing on $[\underline u,\overline u]$. Let $\mu>0$ so large that
  \begin{displaymath}
    M\leq\frac{a+\mu}{2b}.
  \end{displaymath}
  Since $f(\xi)=(a+\mu)\xi-b\xi^2$ is increasing on $\left[0,\frac{a+\mu}{2b}\right]$ it follows that $F_\mu(v):=F(v)+\mu v$ is increasing on $[\underline u,\overline u]$. Now let $0\leq u_0\in E\cap L^\infty(\Omega)$. Then $u_0\in[0,M1_\Omega]$ for $M$ large enough and so a solution of \eqref{eq:parabolic-logistic} exists by Theorem~\ref{thm:existence}. If $u_0\geq\varepsilon\varphi_0$ for some $\varepsilon\in(0,\varepsilon_0]$, then $\lim_{t\to\infty}u(t)=u_\infty$ exists in $E$ and $u_\infty$ is an equilibrium by Theorem~\ref{thm:existence}.
\end{proof}

In the remainder of this subsection we consider $E=C_0(\Omega)$ which needs special attention since $1_\Omega\not\in C_0(\Omega)$. We will consider two operators, the Dirichlet Laplacian (Example~\ref{ex:A-dirichlet-laplacian}) and an elliptic operator in non-divergence form on $C_0(\Omega)$ (Example~\ref{ex:A-non-divergence}).

\begin{theorem}
  \label{thm:logistic-C-dirichlet}
  Assume that $\Omega\subseteq\mathbb R^N$ is Wiener regular and that $E=C_0(\Omega)$. Denote by $A$ the negative Dirichlet Laplacian on $C_0(\Omega)$ as given in Example~\ref{ex:A-dirichlet-laplacian} and let $a>\lambda_1(A)$. Then \eqref{eq:parabolic-logistic} has a unique equilibrium $0<u_\infty\in D(A)$. Moreover, $u_\infty(x)>0$ for all $x\in\Omega$. Denote by $\varphi_0$ the principal eigenvector of $A$. If $u_0\in C_0(\Omega)$ such that
  \begin{displaymath}
    \varepsilon\varphi_0\leq u_0\leq \frac{1}{\varepsilon}u_\infty
  \end{displaymath}
  for some $\varepsilon>0$ small enough, then the logistic parabolic equation \eqref{eq:parabolic-logistic} has a unique mild solution with initial value $u(0)=u_0$ and
  \begin{displaymath}
    \lim_{t\to\infty}u(t)=u_\infty
  \end{displaymath}
  in $C_0(\Omega)$.
\end{theorem}
\begin{proof}
  By Theorem~\ref{thm:logistic-existence} there exists a unique equilibrium $0<u_\infty\in H_0^1(\Omega)\cap L^\infty(\Omega)$ of \eqref{eq:parabolic-logistic}. Since $\Omega$ is Wiener regular it follows from \cite[Proposition~2.9]{arendt:19:dpm} or \cite[proof of part (a) of Theorem~3.8]{arendt:99:wrh} that $u_\infty\in C_0(\Omega)$. Thus $u_\infty\in D(A)$ and $(A+bu_\infty)u_\infty=au_\infty$. By Theorem~\ref{thm:pev} $u_\infty/\|u_\infty\|$ is the principal eigenvector of the operator $A+bu_\infty$ and thus $u_\infty\gg 0$ in $C_0(\Omega)$. It follows from Theorem~\ref{thm:logistic-existence} that $u_\infty$ is the unique equilibrium.

  It follows from the proof of Theorem~\ref{thm:logistic-existence} that $\varepsilon\varphi_0$ is a sub-solution of \eqref{eq:stationary-logistic} if $\varepsilon>0$ is small enough. Let $c\geq 1$. Then
  \begin{displaymath}
    A(cu_\infty)=cau_\infty-cbu_\infty^2\geq a(cu_\infty)-b(cu_\infty)^2.
  \end{displaymath}
  Thus $cu_\infty$ is a super-solution of \eqref{eq:stationary-logistic}. Now it follows from Theorem~\ref{thm:existence} that if $u_0\in C_0(\Omega)$ such that
  \begin{equation}
    \label{eq:bound-phi-u}
    \varepsilon\varphi_0\leq u_0\leq\frac{u_\infty}{\varepsilon}
  \end{equation}
  for some $\varepsilon>0$ small enough, then there exists a unique solution $u$ of \eqref{eq:parabolic-logistic} such that $u(0)=u_0$. Theorem~\ref{thm:existence} also implies that $\lim_{t\to\infty}u(t)=u_\infty$ in $C_0(\Omega)$.
\end{proof}
Next consider an operator in non-divergence form as in Example~\ref{ex:A-non-divergence} on $C_0(\Omega)$, where $\Omega$ satisfies an exterior cone condition.

\begin{theorem}
  \label{thm:logistic-C-dirichlet-nd}
  Assume that $\Omega\subseteq\mathbb R^N$ satisfies an exterior cone condition and that $E=C_0(\Omega)$. Denote by $A$ the elliptic differential operator non-divergence form on $C_0(\Omega)$ as given in Example~\ref{ex:A-non-divergence} and let $a>\lambda_1(A)$. Then \eqref{eq:parabolic-logistic} has a unique equilibrium $0<u_\infty\in D(A)$. Moreover, $u_\infty(x)>0$ for all $x\in\Omega$. Denote by $\varphi_0$ the principal eigenvector of $A$. If $u_0\in C_0(\Omega)$ such that \eqref{eq:bound-phi-u} holds for some $\varepsilon>0$ small enough, then the logistic parabolic equation \eqref{eq:parabolic-logistic} has a unique mild solution with initial value $u(0)=u_0$ and
  \begin{displaymath}
    \lim_{t\to\infty}u(t)=u_\infty
  \end{displaymath}
  in $C_0(\Omega)$.
\end{theorem}
It seems not known whether an elliptic operator in non-divergence form with Dirichlet boundary conditions generates a $C_0$-semigroup on $L^p(\Omega)$ if $\Omega$ merely satifies the uniform exterior cone condition and not higher regularity as in the paper \cite{denk:03:rfm} for example.

\begin{proof}[Proof of Theorem~\ref{thm:logistic-C-dirichlet-nd}]
  Since $1_\Omega\not\in C_0(\Omega)=E$, no obvious super-solution is available for \eqref{eq:stationary-logistic}. To overcome this problem we augment the space and consider the ordered Banach space
  \begin{displaymath}
    \tilde E:=\left\{u\in C(\Omega)\colon\lim_{x\to\partial\Omega}u(x)\text{ exists}\right\}
    =C_0(\Omega)\oplus\mathbb R1_\Omega.
  \end{displaymath}
  We extend the semigroup $(S(t))_{t\geq 0}$ to $\tilde E$ by letting
  \begin{displaymath}
    \tilde S(t)(v+c1_\Omega):=S(t)+c1_\Omega.
  \end{displaymath}
  Then, $(\tilde S(t))_{t\geq 0}$ is a sub-markovian, compact $C_0$-semigroup. However, it is no longer irreducible. Denote by $-\tilde A$ the generator of $(\tilde S(t))_{t\geq 0}$. Then $1_\Omega\in D(\tilde A)$ and $\tilde A1_\Omega=0$.

  Let $\varphi_0$ be the principal eigenvector of $A$. Then, as in the proof of Theorem~\ref{thm:logistic-existence}, $\underline u=\varepsilon\varphi_0$ is a sub-solution of \eqref{eq:stationary-logistic} if $\varepsilon>0$ is small enough. This is also a sub-solution with respect to $\tilde A$. For $c$ large, $\overline u=c1_\Omega$ is a super-solution of \eqref{eq:stationary-logistic} with respect to $\tilde A$. Then we have an ordered pair of sub- and super-solutions $\underline u$ and $\overline u]$ in $\tilde E$ with respect to $\tilde A$. Given an initial condition $u_0\in C_0(\Omega)$. Denote by $w_n$ the lower iteration sequence from Proposition~\ref{prop:iterated-sequences}. This sequence is in fact in the space $L_{\loc}^1([0,\infty),E)$. It converges in $L_{\loc}^1([0,\infty),\tilde E)$ by Theorem~\ref{thm:convergence-iteration}. It follows that $u_{\min}\in L_{\loc}^1([0,\infty),E)$. Since by Theorem~\ref{thm:convergence-iteration} $u_{\min}$ is a mild solution of the parabolic equation \eqref{eq:parabolic-logistic} it follows that $u_{\min}\in C([0,\infty),E)$. By theorem \ref{thm:asymptotics} $u_\infty=\lim_{t\to\infty}u_{\min}(t)$ exists in $\tilde E$, $\tilde Au_\infty=F(u_\infty)$. Since $E$ is a closed subspace of $E$ it follows that $u_\infty\in E$ and hence $u_\infty\in D(A)$.

  If $u_0$ satisfies \eqref{eq:bound-phi-u}, then as in the previous theorem we may choose $u_\infty/\varepsilon$ as a super-solution with respect to $E$ and apply Theorem~\ref{thm:existence} to prove the last claim.
\end{proof}
Results under more general conditions on the non-linearities involving the above classes of operators can be found in \cite{arendt:23:see}. Results using the classical Laplace operator or more general elliptic operators appear in \cite{daners:18:gdg,du:06:ost,lopez:16:mpe,fraile:eep} and many more.
\subsection{A Lotka-Volterra competition model}
Our second example is a two-species Lotka-Volterra competition system. It fits our framework with a non-standard order on a product space as introduced in \cite{hess:91:acm}. Let $E=L_p(\Omega)$, $1\leq p<\infty$, or $E=C(\bar\Omega)$ and let $A_1$ and $A_2$ be two admissible operators. We then consider the Lotka-Volterra competition system
\begin{equation}
  \label{eq:competition-system}
  \dot{\boldsymbol u}(t)+\boldsymbol A\boldsymbol u=F(\boldsymbol u)
\end{equation}
where
\begin{equation*}
  \boldsymbol A:=
  \begin{bmatrix}
    A_1 & 0 \\0 & A_2
  \end{bmatrix}
  \qquad\text{and}\qquad
  F(\boldsymbol u):=
  \begin{bmatrix}
    a_1u_1-b_{11}u_1^2-b_{12}u_1u_2 \\
    a_2u_2-b_{21}u_1u_2-b_{22}u_2^2 \\
  \end{bmatrix}
  .
\end{equation*}
The domain of $\boldsymbol A$ is
$D(\boldsymbol A):=D(A_1)\times D(A_2)$. For simplicity we assume that $a_k>0$ and $b_{kj}>0$ for $k,j\in\{1,2\}$ are constants. We start by proving an existence result for solutions of \eqref{eq:competition-system}.

\begin{proposition}
  \label{prop:cs-existence}
  There exists $M_0>0$ such that for every $M\geq M_0$ and for every $u_0=(u_{01},u_{02})\in E_+\times E_+$ with $0\leq u_{0k}\leq M$, $k=1,2$, there exists a unique mild solution $u=(u_1,u_2)$ of \eqref{eq:competition-system} such that $0\leq u_{k}(t)\leq M$, $k=1,2$.
\end{proposition}
To be able to apply our theory we use a non-standard positive cone on $\boldsymbol E:=E\times E$. As in Hess and Lazer \cite{hess:91:acm} we define $\boldsymbol E_+:=E_+\times(-E_+)$. This means that $\boldsymbol u=(u_1,u_2)\leq(v_1,v_2)=\boldsymbol v$ if and only if $u_1\leq v_1$ and $v_2\leq u_2$. This seems a natural order for a competition system since gains by one species comes at a cost to the other. The operator $\boldsymbol A$ generates a positive compact $C_0$-semigroup on $\boldsymbol E$ given by
\begin{equation*}
  \boldsymbol S(t):=
  \begin{bmatrix}
    T_1(t) & 0 \\0 & T_2(t)
  \end{bmatrix}
  .
\end{equation*}
We continue by showing that $F$ is quasi-monotone with respect to the new order on $\boldsymbol E$. Given $\boldsymbol u,\boldsymbol v\in\boldsymbol E$ we can write
\begin{equation}
  \label{eq:F-qm}
  F(\boldsymbol v)-F(\boldsymbol u)=\Phi(u,v)(\boldsymbol v-\boldsymbol u),
\end{equation}
where
\begin{equation*}
  \Phi(u,v):=
  \begin{bmatrix}
    a_1-b_{11}(u_1+v_1)-b_{12}v_2 & -b_{12}u_1                    \\
    -b_{21}u_2                    & a_2-b_{21}v_1-b_{22}(u_2+v_2) \\
  \end{bmatrix}
  .
\end{equation*}
This means that $\Phi(\boldsymbol u,\boldsymbol v)$ has the form
\begin{equation*}
  \begin{bmatrix}
    c_{11}  & -c_{12} \\
    -c_{21} & c_{22}
  \end{bmatrix}
\end{equation*}
with $c_{12},c_{21}\geq 0$.  If $\boldsymbol w=(w_1,-w_2)\in \boldsymbol E_+$, then
\begin{equation*}
  \Phi(u,v)\boldsymbol w+\mu\boldsymbol w=
  \begin{bmatrix}
    c_{11}+\mu & -c_{12}    \\
    -c_{21}    & c_{22}+\mu \\
  \end{bmatrix}
  \begin{bmatrix}
    w_1 \\-w_2
  \end{bmatrix}
  =
  \begin{bmatrix}
    (c_{11}+\mu)w_1+c_{12}w_2 \\
    -c_{21}w_1-(c_{22}+\mu)w_2
  \end{bmatrix}
  .
\end{equation*}
If we assume that $\|\boldsymbol u\|_\infty+\|\boldsymbol v\|_\infty\leq M$ for some bound $M>0$, then we can choose $\mu>0$ such that $c_{11}+\mu>0$ and $c_{22}+\mu>0$ and thus $\Phi(\boldsymbol u,\boldsymbol v)(\boldsymbol v-\boldsymbol u)\geq 0$ with respect to the order in $\boldsymbol E$. This shows that $F$ is quasi-increasing on any bounded set in $L^\infty(\Omega)\times L^\infty(\Omega)$.

\begin{proof}[Proof of Proposition~\ref{prop:cs-existence}]
  We construct some sub- and super-solutions. As $b_{11},b_{22}>0$ there exists a constant $M_0>0$ such that $a_kM-b_{kk}M^2<0$ for all $M\geq M_0$ ($k=1,2$). We set
  \begin{equation*}
    \underline{\boldsymbol u}:=
    \begin{bmatrix}
      0 \\M1_\Omega
    \end{bmatrix}
    ,
    \qquad
    \overline{\boldsymbol u}:=
    \begin{bmatrix}
      M1_\Omega \\0
    \end{bmatrix}
  \end{equation*}
  and claim that they form an ordered pair of weak sub- and super-solutions of \eqref{eq:competition-system}. First note that
  \begin{equation*}
    \overline{\boldsymbol u}-\underline{\boldsymbol u}
    =
    \begin{bmatrix}
      M1_\Omega \\-M1_\Omega
    \end{bmatrix}
    \in\boldsymbol E_+.
  \end{equation*}
  by definition of the order on $\boldsymbol E$ and thus $\underline{\boldsymbol u}\leq\overline{\boldsymbol u}$. We next show $\overline{\boldsymbol u}$ is a weak super-solution. To do so take $\boldsymbol\varphi=[\varphi_1,-\varphi_2]\in \boldsymbol E_+'\cap D(\boldsymbol A')$ with $\varphi_1,\varphi_2\in E_+'$. We note that the second component of $F(\overline{\boldsymbol u})$ vanishes. As $M1_\Omega$ is a weak super-solution of $A_1u=0$ we see that
  \begin{equation*}
    \langle\overline{\boldsymbol u}, \boldsymbol A'\boldsymbol\varphi\rangle
    =\langle M1_\Omega, A_1\varphi_1\rangle
    \geq 0
    \geq\langle a_1M1_\Omega-b_{11}M^21_\Omega, \varphi_1\rangle
    =\langle F(\overline{\boldsymbol u}), \boldsymbol\varphi\rangle.
  \end{equation*}
  Similarly, as the first component of $F(\underline{\boldsymbol u})$ vanishes, we see that
  \begin{equation*}
    \langle\underline{\boldsymbol u}, \boldsymbol A'\boldsymbol\varphi\rangle
    =-\langle M1_\Omega, A_2\varphi_2\rangle
    \leq 0
    \leq\langle a_2M1_\Omega-b_{22}M^21_\Omega, -\varphi_2\rangle
    =\langle F(\underline{\boldsymbol u}),\boldsymbol\varphi\rangle.
  \end{equation*}
  This shows that $\underline{\boldsymbol u}$ and $\overline{\boldsymbol u}$ is an ordered pair of sub- and super-solutions. As $F$ is quasi-monotone the claim follows from Theorem~\ref{thm:existence}. The identity \eqref{eq:F-qm} also implies that $F$ is Lipschitz continuity on $[\underline{\boldsymbol u},\overline{\boldsymbol u}]$ and thus the solution is unique by Proposition~\ref{prop:uniqueness}.
\end{proof}
We consider the stationary problem
\begin{equation}
  \label{eq:stationary-competition}
  \boldsymbol A\boldsymbol u=F(\boldsymbol u)
\end{equation}
We first look at the case where one species is absent. The states $(w_1,0)$ and $(0,w_2)$ with $0<w_k\in D(A_k)$, $k=1,2$, are equilibrium solutions of \eqref{eq:competition-system} if and only if $w_1$, $w_2$ are solutions to the logistic equations
\begin{equation}
  \label{eq:semi-trivial-logistic}
  A_1w_1=a_1w_1-b_{11}w_1^2
  \qquad\text{and}\qquad
  A_2w_2=a_2w_2-b_{22}w_2^2,
\end{equation}
respectively. We call $(w_1,0)$ and $(0,w_2)$ the \emph{semi-trivial solutions} of \eqref{eq:stationary-competition}. Section~\ref{ex:logistic-equation} implies that
\begin{equation}
  \label{eq:cs-necessary}
  a_1>\lambda_1(A_1)\qquad\text{and}\qquad a_2>\lambda_1(A_2)
\end{equation}
is a necessary condition for the existence of a non-trivial equilibrium solution of \eqref{eq:stationary-competition}. For the existence of a \emph{coexistence state}, that is, a stationary solution $(u_1,u_2)$ of \eqref{eq:stationary-competition} with $u_1>0$ and $u_2>0$ we need a stronger condition.

\begin{theorem}[Coexistence in competition systems]
  \label{thm:coexistence}
  (a) Assume that a coexistence state $(u_1,u_2)$ exists. Then \eqref{eq:cs-necessary} holds. Moreover,
  \begin{equation}
    \label{eq:coexistence-estimates}
    0\leq u_1\leq w_1\quad\text{and}\quad 0\leq u_2\leq w_2,
  \end{equation}
  where $(w_1,0)$ and $(0,w_2)$ are the semi-trivial solutions of \eqref{eq:stationary-competition}.

  (b) If we assume that
  \begin{equation}
    \label{eq:cs-sufficient}
    a_1>\lambda_1(A_1+b_{12}w_2)
    \qquad\text{and}\qquad
    a_2>\lambda_1(A_2+b_{21}w_1),
  \end{equation}
  then there exists a coexistence state. More precisely, there exist a sub-solution and a super-solution with respect to the new order on $\boldsymbol E$ such the assertions of Theorem~\ref{thm:existence} are valid for $\boldsymbol A$ and $F$.
\end{theorem}

Recall that \eqref{eq:cs-necessary} is necessary for the existence of a coexistence state. We can only prove that the stronger condition \eqref{eq:cs-sufficient} is sufficient. We further note that if \eqref{eq:cs-necessary} is satisfied, then \eqref{eq:cs-sufficient} is satisfied if $b_{12}$ and $b_{21}$ are small enough.

The above theorem provides a sufficient condition for the existence of a coexistence state for a rather general class of operators.

For the Laplace operator more about the structure of coexistence states is known. In particular, for instance the results in \cite{blat:84:bss,dancer:85:psp,dancer:84:psp,du:94:bmc,eilbeck:94:ccm,hess:91:ppb} show that the solution structure can be quite complicated.  Not all solutions can be obtained by means of the method of sub- and super-solutions. In particular the non-stable ones cannot and other methods such as bifurcation or fixed point index calculations are used, see the above references.

Property~\eqref{eq:coexistence-estimates} has a natural interpretation. A coexistence state has to be below an equilibrium without competition.

\begin{proof}[Proof of Theorem~\ref{thm:coexistence}]
  (a) We first prove that \eqref{eq:cs-necessary} is a necessary condition for the existence of a coexistence state and that \eqref{eq:coexistence-estimates} holds. Given a coexistence solution $\boldsymbol u=(u_1,u_2)$ we have
  \begin{equation*}
    A_1u_1=a_1u_1-b_{11}u_1^2-b_{12}u_1u_2\leq a_1u_1-b_{11}u_1^2.
  \end{equation*}
  Hence, $u_1>0$ is a sub-solution of the first equation in \eqref{eq:semi-trivial-logistic}. As seen before, we can also choose a constant $M\geq u_1$ such that $M1_\Omega$ is a weak super-solution of the logistic equation. It follows from Theorem~\ref{thm:existence} that the logistic equation $A_1u=a_1u-b_{11}u^2$ has an equilibrium in that interval. As the non-zero solution of the logistic equation is unique it coincides with $w_1$ and hence $0<u_1\leq w_1$. A similar argument applies to $u_2$. Since
  \begin{displaymath}
    (A_1+b_11u_1+b_{12}u_2)u_1=a_1u,
  \end{displaymath}
  it follows from Theorem~\ref{thm:pev} and Theorem~\ref{thm:ev-monotone} that
  \begin{displaymath}
    \lambda_1(A_1)<\lambda_1(A_1+b_11u_1+b_{12}u_2)=a_1.
  \end{displaymath}
  Hence, \eqref{eq:cs-necessary} is a necessary condition for the existence of a non-trivial equilibrium solution of \eqref{eq:competition-system}.

  (b) We now assume that \eqref{eq:cs-sufficient} holds. Let $v_1$ and $v_2$ by the principal eigenvectors associated with $A_1+b_{12}w_2$ and $A_2+b_{21}w_1$, respectively. We show that
  \begin{equation*}
    \underline{\boldsymbol u} = (\varepsilon v_1,w_2),
    \qquad
    \overline{\boldsymbol u} = (w_1,\varepsilon v_2)
  \end{equation*}
  form a pair of ordered sub- and super-solutions. Indeed, note that
  \begin{align*}
    A_1(\varepsilon v_1)
    & =\lambda_1(A_1+b_{12}w_2)(\varepsilon v_1)-b_{12}w_2(\varepsilon v_1)                     \\
    & =a_1(\varepsilon v_1)-b_{11}(\varepsilon v_1)^2-b_{12}w_2(\varepsilon v_1)                \\
    & \qquad+\bigl[\lambda_1(A_1+b_{12}w_2)-a_1 +b_{11}(\varepsilon v_1)\bigr](\varepsilon v_1)
  \end{align*}
  for all $\varepsilon>0$. As $v_1$ is bounded and $\lambda_1(A_1+b_{12}w_2)-a_1<0$ by assumption there exists $\varepsilon>0$ such that
  \begin{equation}
    \label{eq:v1-w2-subsolution}
    A_1(\varepsilon v_1)
    <a_1(\varepsilon v_1)-b_{11}(\varepsilon v_1)^2-b_{12}w_2(\varepsilon v_1)
  \end{equation}
  for all $\varepsilon\in(0,\varepsilon_0)$. We furthermore have that
  \begin{equation*}
    A_2w_2=a_2w_2-b_{22}w_2^2
    >a_2w_2-b_{21}(\varepsilon v_1)w_2-b_{22}w_2^2
  \end{equation*}
  for all $\varepsilon\in(0,\varepsilon_0]$. By definition of the positive cone in $\boldsymbol E$ we see that $\boldsymbol A\underline{\boldsymbol u}<F(\underline{\boldsymbol u})$, showing that $\underline{\boldsymbol u}$ is a sub-solution whenever $\varepsilon\in(0,\varepsilon_0]$. We can do similar calculations with the roles of the two equations interchanged. We then see that $\overline{\boldsymbol u}$ is a super-solution if $\varepsilon\in(0,\varepsilon_0]$ by possibly making $\varepsilon_0$ smaller. We next show that the pair of sub- and super-solutions can be ordered. We need to make sure that $\varepsilon v_1\leq w_1$ and $\varepsilon v_2\leq w_2$. It follows from \eqref{eq:v1-w2-subsolution} that
\begin{equation*}
  A_1(\varepsilon v_1)
  <a_1(\varepsilon v_1)-b_{11}(\varepsilon v_1)^2,
\end{equation*}
that is, $\varepsilon v_1$ is a sub-solution of the logistic equation $A_1w_1=a_1w_1-b_{11}w_1^2$. As before, it follows that $\varepsilon\varphi_1\leq w_1$ for all $\varepsilon\in(0,\varepsilon_0]$. A similar argument shows that $\varepsilon v_2\leq w_2$ for all $\varepsilon\in(0,\varepsilon_0]$. The choice given here will lead to the existence of a coexistence state due to Theorem~\ref{thm:existence}.
\end{proof}

We note that the conditions \eqref{eq:cs-sufficient} are often formulated in terms of spectral radii of some operators, see \cite{dancer:85:psp,dancer:84:psp}. In our context is more convenient to use conditions in terms of principal eigenvalues adapted from \cite[Theorem~4.1]{eilbeck:94:ccm}.

\subsection{The Fisher Equation}
As a last example we apply our theory to a simple version of the Fisher equation from population genetics as proposed by \cite{fleming:75:smm}. The Fisher equation models the evolution of two alleles $A_1$, $A_2$ corresponding to the genotypes $A_1A_1$, $A_1A_2$ and $A_2A_2$. Denote by $u$ the proportion of the allele $A_1$ at time $t$ and location $x$ in the habitat $\Omega\subseteq\mathbb R^N$. The change of alleles through selection and diffusion is modelled by a semi-linear equation of the abstract form
\begin{equation}
  \label{eq:fisher-equation}
  \dot u+Au=mh(u)\qquad\text{for $t>0$}
\end{equation}
where $m\in L^\infty(\Omega)$ and $h\colon\mathbb R\to\mathbb R$ is given by
\begin{equation*}
  \xi\mapsto h(\xi):=\xi(1-\xi)\bigl(\alpha(1-\xi)+(1-\alpha)\xi\bigr)
\end{equation*}
for some $\alpha\in(0,1)$. In the original setup, $A$ is the Neumann Laplacian, but we will work with any of the admissible operators on $L^p(\Omega)$, $1\leq p<\infty$ or $C(\bar\Omega)$ from Section~\ref{sec:operators}. The parameter $\alpha$ determines the fitness of the three genotypes in the proportions $1:1-\alpha m:1-m$. The weight $m$ may change sign. For more precise explanations see \cite{fleming:75:smm} or \cite[Section~29]{hess:91:ppb}. The corresponding superposition operator $u\mapsto h\circ u$ is Lipschitz continuous and quasi-monotone on every bounded set in $L^\infty(\Omega)$ by Remark~\ref{rem:superposition-operator}.

Since $u$ is a fraction of a population we are only interested in solutions with $0\leq u\leq 1_\Omega$. Hence we make sure that the constant functions $u=0$ and $u=1_\Omega$ are a pair of weak sub- and super-solutions. We note that $h(0)=h(1)=0$. In particular, the constant function $u=0$ is always a solution. Since $A$ is an admissible operator
\begin{equation*}
  \langle 1_\Omega,A'v'\rangle\geq 0,
\end{equation*}
for all $0\leq v'\in D(A')$. Hence, the constant function $u=1_\Omega$ is a weak super-solution of the stationary equation
\begin{equation}
  \label{eq:eq:fisher-equation-stationary}
  Au=mh(u)
\end{equation}
in $E$. As a consequence of Theorem~\ref{thm:existence} and Proposition~\ref{prop:uniqueness}, for every $u_0\in[0,1_\Omega]$ the parabolic equation \eqref{eq:fisher-equation} has a unique solution $u\in C([0,\infty),E)$ with values in $[0,1_\Omega]$. The equilibria for \eqref{eq:fisher-equation} could be one of $u=0$ or $u=1_\Omega$, so we need additional assumptions to guarantee the existence of non-trivial equilibria.

We show that if $u=0$ is a linearly unstable solution of \eqref{eq:fisher-equation}, then the there exits an equilibrium $0<u_*\leq 1_\Omega$. A short computation shows that $h'(0)=\alpha$. Hence the eigenvalue problem associated with the linearization of \eqref{eq:eq:fisher-equation-stationary} at $u=0$ is
\begin{equation*}
  Av-\alpha m v=\lambda v.
\end{equation*}
Let $\lambda_0:=\lambda_1(-\alpha m)$ be its principal eigenvalue. Denote the corresponding principal eigenvector by $\varphi_0$. The zero solution of \eqref{eq:fisher-equation} is linearly unstable if and only if $\lambda_0<0$. Assume that $\lambda_1<0$. Since $h'(0)=\alpha$ and $h(0)=0$ we can write
\begin{equation*}
  h(\xi)=\alpha\xi+r_0(\xi)\xi
\end{equation*}
for all $\xi\in\mathbb R$ where $r_0\in C(\mathbb R)$ with $r_0(0)=0$. Hence, for $\varepsilon>0$
\begin{equation*}
  A(\varepsilon\varphi_0)=\lambda_0\varepsilon\varphi_0+\alpha m\varepsilon\varphi_0
  =m h(\varepsilon\varphi_0)+\varepsilon\varphi_0\bigl(\lambda_0-mr_0(\varepsilon\varphi_0)\bigr)
\end{equation*}
As $\varphi_0\in L^\infty(\Omega)$, $\lambda_0<0$ and $r_0(0)=\alpha$ there exists $\varepsilon_0>0$ such that
\begin{equation*}
  \lambda_0-mr_0(\varepsilon\varphi)<0
\end{equation*}
for all $\varepsilon\in(0,\varepsilon_0]$. Hence $\varepsilon\varphi_0$ is a sub-solution of \eqref{eq:eq:fisher-equation-stationary} for all $\varepsilon\in(0,\varepsilon_0]$.

We show that if $u=1_\Omega$ is a linearly unstable solution of \eqref{eq:fisher-equation}, then the there is a stationary solution $0\leq u^*<1_\Omega$. A short computation shows that $h'(1)=1-\alpha$. Hence the eigenvalue problem associated with the linearization of \eqref{eq:eq:fisher-equation-stationary} at $u=1_\Omega$ is
\begin{equation*}
  Av-(1-\alpha)m v=\lambda v.
\end{equation*}
 Let $\lambda_1:=\lambda_1(-(1-\alpha)m))<0$ be its principal eigenvalue. Denote the corresponding principal eigenvector by $\varphi_1$. The zero solution of \eqref{eq:fisher-equation} is linearly unstable if and only if $\lambda_1<0$. Assume now that $\lambda_1<0$. Since $h'(1)=1-\alpha$ and $h(1)=0$ we can write
\begin{equation*}
  h(\xi)=(\alpha-1)(\xi-1)+r_1(\xi)(\xi-1)
\end{equation*}
for all $\xi\in\mathbb R$ where $r_1\in C(\mathbb R)$ with $r_1(1)=0$. Hence, as $1_\Omega$ is a weak super-solution for $Au=0$, for $\varepsilon>0$ and $\leq v'\in D(A')$,
\begin{align*}
  \langle 1_\Omega-\varepsilon\varphi_1,A'v'\rangle
   & \geq\langle-\varepsilon\varphi_1,A'v'\rangle                                                                          \\
   & =-\bigl\langle\bigl((\alpha-1)m+\lambda_1\bigr)\varepsilon\varphi_1,v'\bigr\rangle                                    \\
   & =\bigl\langle mh(1_\Omega-\varepsilon\varphi_1)+\varepsilon\varphi_1(mr_1(1_\Omega-\varepsilon\varphi_1)-\lambda_1),v'\bigr\rangle.
\end{align*}
for all $\varepsilon>0$. As $\varphi_1\in L^\infty(\Omega)$, $r_1(1)=0$ and $-\lambda_1>0$, there exists $\varepsilon_1>0$ such that
\begin{equation*}
  \langle 1_\Omega-\varepsilon\varphi_1,A'v'\rangle
  \geq\bigl\langle mh(1-\varepsilon\varphi_1),v'\bigr\rangle
\end{equation*}
for all $\varepsilon\in(0,\varepsilon_1)$ and all $0\leq v'\in D(A')$. Hence $1_\Omega-\varepsilon\varphi_1$ is a weak super-solution of \eqref{eq:eq:fisher-equation-stationary} for all $\varepsilon\in(0,\varepsilon_1]$.

Assume now that both $0$ and $1_\Omega$ are linearly unstable solutions. Since $\|\varepsilon\varphi_0\|_\infty\to 0$ and $\|1-\varepsilon\varphi_1\|_\infty\to 1$ we can find $\varepsilon>0$ such that $\varepsilon\varphi_0<1_\Omega-\varepsilon\varphi_1$, that is, we have a pair of ordered sub-and super-solutions to which Theorem~\ref{thm:existence} applies. As pointed out in \cite[Section~29, p.~99]{hess:91:ppb} such a situation is only possible if $m$ changes sign.

\appendix

\section{Appendix: Some facts on ordered Banach spaces}
In this appendix we recall and prove some facts on ordered Banach spaces which are useful for our purposes.
\begin{proposition}
  \label{prop:weak-positivity}
  Let $E$ be an ordered Banach space and let $u\in E$. If $\langle v',u\rangle\geq 0$ for all $v'\in E_+'$, then $u\geq 0$.
\end{proposition}
\begin{proof}
  Consider the semi-linear mapping $p\colon E\to\mathbb R$ given by $p(v):=\inf_{w\in E_+}\|v-w\|_E$. Since $0\in E_+$ we have $0\leq p(v)\leq \|v\|$ for all $v\in E$. Consequently
  \begin{equation*}
    |p(v)|\leq \|v\|
  \end{equation*}
  for all $v\in E$. As $E_+$ is closed we also have that
  \begin{equation*}
    p(v)=0\iff v\geq 0.
  \end{equation*}
  Let now $u\in E$ be such that $\langle v',u\rangle\geq 0$ for all $v'\in E_+'$. By the Hahn-Banach Theorem there exists a linear map $\varphi\colon E\to \mathbb R$ such that $\langle\varphi,u\rangle=p(u)$ and $\langle\varphi,v\rangle\leq p(v)$ for all $v\in E$. Thus
  \begin{equation*}
    \pm\langle\varphi,v\rangle\leq p(\pm v)\leq\|v\|
  \end{equation*}
  for all $v\in E$. It follows that $\|\varphi\|\leq 1$. Moreover, if $v\geq 0$, then $\langle\varphi,v\rangle\leq p(v)=0$ and thus $-\varphi\geq 0$. Therefore $p(u)=\langle\varphi,u\rangle=-\langle-\varphi,u\rangle\leq 0$ by our assumption. Thus $p(u)=0$, which means that $u\geq 0$.
\end{proof}
We next show that $D(A')_+:=D(A')\cap E_+'$ determines positivity if $-A$ is the generator of a positive semigroup.
\begin{corollary}
  \label{cor:weak-positivity}
  Let $E$ be an ordered Banach space with normal cone and let $-A$ be the generator of a positive $C_0$-semigroup on $E$. Suppose that $u\in E$ is such that $\langle v',u\rangle\geq 0$ for all $v'\in D(A')_+$. Then $u\geq 0$.
\end{corollary}
\begin{proof}
  Let $w'\in E_+'$. Using the Yosida approximation, see for instance \cite[Section~III.4.10]{engel:00:ops}, we have that $\lambda\bigl((\lambda+A')^{-1}\bigr)'w'\in D(A')_+$ for all $\lambda$ large enough and
  \begin{equation*}
    \langle w',u\rangle
    =\lim_{\lambda\to\infty}\langle\lambda(\lambda+A')^{-1}w',u\rangle
    \geq 0.
  \end{equation*}
  Now Proposition~\ref{prop:weak-positivity} implies that $u\geq 0$.
\end{proof}
We say that the positive cone $E_+$ is \emph{generating} if $E=E_+-E_+$. As a consequence there exists $\gamma>0$ such that for all $u\in E$ there exist $u_1,u_2\in E_+$ such that
\begin{equation}
  \label{eq:normal-cone-estimate}
  u=u_1-u_2\text{ and }\|u_1\|+\|u_2\|\leq\gamma\|u\|,
\end{equation}
see \cite[Lemma~2.2]{arendt:09:ecn}. We let
\begin{equation*}
  B_+':=\{v'\in E_+'\colon \|v'\|\leq 1\}.
\end{equation*}
We next give characterisations of ordered Banach spaces with a normal cone.
\begin{lemma}
  \label{lem:normal-cone}
  Suppose that $E$ is an ordered Banach space. Then the following assertions are equivalent.
  \begin{enumerate}[\normalfont (i)]
  \item $E_+$ is a normal cone;%
  \item $E_+'$ is generating;%
  \item $\tnorm{u}:=\sup_{v'\in B_+'}|\langle v',u\rangle|$ for $u\in E$ defines an equivalent norm on $E$.
  \end{enumerate}
\end{lemma}
Note that the norm $\tnorm{{\cdot}}$ is \emph{monotone}, that is $0\leq u\leq v$ implies that $\tnorm{u}\leq\tnorm{v}$.
\begin{proof}
  $(i)\implies(ii)$: See \cite[Theorem~2.26]{aliprantis:07:cd}, where a proof is given for ordered locally convex spaces.

  $(ii)\implies(iii)$: Let $v\in E$. By definition of the dual norm we have $\tnorm{v}\leq\|v\|$. To estimate the other direction use the Hahn-Banach Theorem to choose $v'\in E'$ such that $\|v'\|=1$ and $\langle v',v\rangle=1$. Using \eqref{eq:normal-cone-estimate} we find $v_1',v_2'\in E_+'$ such that $v'=v_1'-v_2'$ and $\|v_1'\|+\|v_2\|\leq \gamma$. Thus
  \begin{equation*}
    \|v\|=\langle v',v\rangle
    =\langle v_1'-v_2',v\rangle
    \leq |\langle v_1',v\rangle|+|\langle v_2',v\rangle
    \leq \|v_1'\|\tnorm{v}+\|v_2'\|\tnorm{v}
    \leq\gamma\tnorm{v}
  \end{equation*}
  by definition of $\tnorm{{\cdot}}$. Hence $\tnorm{{\cdot}}$ is an equivalent norm on $E$.

  $(iii)\implies(i)$: We may assume without loss of generality that the norm on $E$ is monotone. Let $[a,b]$ be an order interval. If $x\in [a,b]$, then $0\leq x-a\leq b-a$ and thus by the monotonicity of the norm
  \begin{equation*}
    \|x\|=\|a+x-a\|\leq \|a\|+\|x-a\|\leq \|a\|+\|b-a\|\leq 2(\|a\|+\|b\|).
  \end{equation*}
  Hence $[a,b]$ is bounded and thus $E_+$ normal.
\end{proof}
We next consider spaces with order continuous norm.
\begin{examples}
  (a) Let $E$ be an ordered Banach space such that there exists $p\in[1,\infty)$ with
  \begin{equation}
    \label{eq:p-identity}
    \|u+v\|^p=\|u\|^p+\|v\|^p
  \end{equation}
  for all $u,v\in E_+$. Then $E$ has order continuous norm.

  Let $(x_n)_{n\geq 1}$ be a sequence in $E$ with $0\leq x_n\leq x_{n+1}\leq b$ for all $n\in\mathbb N$. Letting $x_0:=0$ we then have
  \begin{equation*}
    x_{n+1}=(x_1-x_0)+(x_2-x_1)+(x_3-x_2)+\dots+(x_{n+1}-x_n)
  \end{equation*}
  for all $n\in\mathbb N$. It follows that
  \begin{align*}
    \sum_{k=0}^n\|x_{k+1}-x_k\|^p
     & =\Bigl\|\sum_{k=0}^n(x_{k+1}-x_k)\Bigr\|^p
    =\|x_{n+1}\|^p                                \\
     & \leq\|x_{n+1}\|^p+\|b-x_{n+1}\|^p
    =\|x_{n+1}+(b-x_{n+1})\|^p
    =\|b\|^p<\infty
  \end{align*}
  for all $n\in\mathbb N$. Choose $n_0\in\mathbb N$ such that
  \begin{equation*}
    \sum_{k=n_0+1}^n\|x_{k+1}-x_k\|^p<\varepsilon^p.
  \end{equation*}
  For $m>n\geq n_0$ we have
  \begin{align*}
    \|x_m-x_n\|^p
    =\sum_{k=n+1}^m\|x_k-x_{k-1}\|^p<\varepsilon^p.
  \end{align*}
  Hence $(x_n)$ is a Cauchy sequence and by the completeness $x_n\to x$ for some $x\in E$.

  (b) By (a) the $L^p$-spaces have order continuous norm for $1\leq p<\infty$, but not for $p=\infty$.

  (c) If $E_+$ is normal and $E$ is reflexive, then $E$ has order continuous norm. In fact, let $0\leq u_n\leq u_{n+1}\leq u$ for all $n\in\mathbb N$. As $(u_n)$ is bounded and $E$ is reflexive there exists a weakly convergent sub-sequence $(u_{n_k})_{k\in\mathbb N}$. The Dini argument in the proof of Proposition~\ref{prop:uniform-convergence} implies that $(u_{n_k})_{k\in\mathbb N}$ converges in norm. Lemma~\ref{lem:monotone-norm-convergence} implies that $(u_n)_{n\in\mathbb N}$ converges.

  (d) If order intervals are weakly compact, then $E$ has order continuous norm. This follows from the argument in (c). A Banach lattice has order continuous norm if and only if order intervals are weakly compact. Note that order intervals are closed and convex without further hypotheses on the ordered Banach space. As a consequence they are always weakly closed.
\end{examples}

\pdfbookmark[1]{\refname}{biblio}%

\end{document}